\documentclass[12pt,leqno]{amsart} 
\usepackage{amsmath,amstext,amsthm,amssymb,amsxtra} 
\usepackage[bookmarksopen,bookmarksdepth=3,colorlinks,citecolor=red,pagebackref,hypertexnames=true]{hyperref} 
\usepackage[backrefs,msc-links,nobysame]{amsrefs} 
\usepackage{txfonts} 
\usepackage[T1]{fontenc}
\usepackage{lmodern} 

\rmfamily

\usepackage{mathtools,MnSymbol} 

\allowdisplaybreaks

\theoremstyle{plain}

\newtheorem{lemma}[equation]{Lemma} 
\newtheorem{proposition}[equation]{Proposition} 
\newtheorem{theorem}[equation]{Theorem} 
\newtheorem{corollary}[equation]{Corollary}
\theoremstyle{definition} 
\newtheorem{definition}[equation]{Definition}
\newtheorem{example}[equation]{Example}
\newtheorem{remark}[equation]{Remark}
\newtheorem{question}[equation]{Question}

\numberwithin{equation}{section}

\def\norm#1.#2.{\lVert#1\rVert_{#2}} 
\def\Norm#1.#2.{\bigl\lVert#1\bigr\rVert_{#2}} 
\def\NOrm#1.#2.{\Bigl\lVert#1\Bigr\rVert_{#2}} 
\def\NORm#1.#2.{\biggl\lVert#1\biggr\rVert_{#2}} 
\def\NORM#1.#2.{\Biggl\lVert#1\Biggr\rVert_{#2}}

\def\Rec{\textnormal{Rec}} 
\def\Orb{\textnormal{Orb}} 
 
\def\Ker{\textnormal{Ker}} 

\def\HC{\textnormal{HC} }

\def\R{\mathbb R} 
\def\C{\mathbb C}

\def\lfm{\textnormal{LFM} (\mathbb D)}

\def\ip#1,#2,{\langle #1,#2\rangle} 
\def\Ip#1,#2,{\bigl\langle#1,#2\bigr\rangle} 
\def\IP#1,#2,{\Bigl\langle#1,#2\Bigr\rangle}

\def\abs#1{\lvert#1\rvert} 
\def\Abs#1{\bigl\lvert#1\bigr\rvert} 
\def\ABs#1{\biggl\lvert#1\biggr\rvert}

\begin{document}
\raggedbottom
\title[Recurrent linear operators]{Recurrent linear operators}

\author[G. Costakis]{George Costakis} \address{G.C: Department of Mathematics, University of Crete, Knossos Avenue, GR-714 09 Heraklion, Crete, Greece.} \email{costakis@math.uoc.gr} \urladdr{\href{http://www.math.uoc.gr:1080/Members/costakis}{http://www.math.uoc.gr:1080/Members/costakis}}
\author[A. Manoussos]{Antonios Manoussos} \address{A.M.: Fakult\"{a}t f\"{u}r Mathematik, SFB 701, Universit\"{a}t Bielefeld, Postfach 100131, D-33501 Bielefeld, Germany} \email{amanouss@math.uni-bielefeld.de} \urladdr{\href{http://www.math.uni-bielefeld.de/~amanouss} {http://www.math.uni-bielefeld.de/~amanouss}}\thanks{A.M. is fully supported by SFB 701 ``Spektrale Strukturen und Topologische Methoden in der Mathematik" at the University of Bielefeld, Germany.}
\author[I. Parissis]{Ioannis Parissis} \address{I.P.: Department of Mathematics, Aalto University, P.O.Box 11100, FI-00076 Aalto, Finland} \email{ioannis.parissis@gmail.com}\urladdr{\href{http://www.math.aalto.fi/~parissi1}{http://www.math.aalto.fi/~parissi1}} \thanks{I.P. is supported by the Academy of Finland, grant 138738.}

\keywords{recurrent operator, rigid operator, unitary, normal, hyponormal operator} \subjclass[2010]{Primary: 47A16 Secondary: 37B20} 
\begin{abstract}
	We study the notion of recurrence and some of its variations for linear operators acting on Banach spaces. We characterize recurrence for several classes of linear operators such as weighted shifts, composition operators and multiplication operators on classical Banach spaces. We show that on separable complex Hilbert spaces the study of recurrent operators reduces, in many cases, to the study of unitary operators. Finally, we study the notion of product recurrence and state some relevant open questions. 
\end{abstract}
\maketitle
\tableofcontents

\section{Introduction} \label{s.intro}

The most studied notion in linear dynamics is that of hypercyclicity: a bounded linear operator $T$ acting on a separable Banach space is hypercyclic if there exists a vector whose orbit under $T$ is dense in the space. On the other hand, a very central notion in topological dynamics is that of recurrence. This notion goes back to Poincar\'e and Birkhoff and it refers to the existence of points in the space for which parts of their orbits under a continuous map ``return'' to themselves. The purpose of this note is the study of the notion of \emph{recurrence}, together with its variations, in the context of linear dynamics. Some examples and characterizations of recurrence for special classes of linear operators have appeared in \cite{CP}. In the present paper we develop the properties of recurrent operators in a more systematic way, and give examples and characterizations for some classes of recurrent operators such as weighted shifts, unitary operators, composition operators and multiplication operators. 

In an effort to characterize recurrent linear operators one many times falls back to the notion of hypercyclicity. This is for example the case when we study the recurrence properties of backwards shifts, say on $\ell^2(\mathbb Z)$. The reason behind is that, according to a result of Seceleanu, \cite{Sece:thesis}, the orbits of these operators satisfy a \emph{zero-one} law: if the orbit of a weighted backward shift contains a non-zero limit point then the corresponding shift is actually hypercyclic. Thus a weighted backward shift on $\ell^2(\mathbb Z)$ is recurrent if and only if it is hypercyclic. The same equivalence is true, albeit for different reasons, for the adjoint of a multiplication operators on the Hardy space $H^2(\mathbb D)$. These connections to hypercyclicity, already observed in \cite{CP}, come up naturally and thus motivate a further search on whether the properties of recurrent operators resemble the properties of hypercyclic ones, in general. It turns out that, indeed, there are many structural similarities between the set of hypercyclic vectors and the set of recurrent vectors in the sense that they exhibit the same invariances. Furthermore, the spectral properties of hypercyclic and recurrent operators are somewhat similar, although this vague statement should be interpreted with some care. However, these similarities cannot be pushed too much as there are obviously many classes of operators which are recurrent without being hypercyclic. One can find such examples among composition operators on the Hardy space $H^2(\mathbb D)$. However, the primordial example is given just by considering unimodular multiples of the identity operator. A more general class for which one needs to address the recurrence properties independently of hypercyclicity is that of unitary operators on Hilbert spaces. 

The discussion above hopefully justifies why we will shortly recall a full set of definitions relating to hypercyclicity, and not just stick to the notions of recurrence which is the main object of this paper.

\subsection*{Notations} We will work on a complex Banach space $X$ and $T:X\to X$ will always denote a bounded linear operator acting on $X$. We will just refer to $T$ as an operator acting on $X$ with the understanding that it is linear and bounded. In several occasions we will need to work with Fr\'echet spaces $Y$ in which case we consider operators $T:Y\to Y$ which are continuous with respect to the topology induced by the (complete invariant) metric. We reserve the notation $H$ for a Hilbert space over the complex numbers. In general, our spaces will be always considered over the complex numbers unless otherwise stated. If $K\subset X$ we write $\overline{K}$ for the closure of $K$ in the norm topology.  We will denote by $B(x,\epsilon)$ the open ball in $X$ of center $x\in X$ and radius $\epsilon>0$ while we write $D(z,r)$ for disks of center $z$ and radius $r>0$ in the complex plane $\mathbb C$. We will write $\mathbb D$ for the open unit disk of the complex plane and $ \mathbb T$ for its boundary. We denote by $\hat{\mathbb C}$ the extended complex plane. Finally, we denote by $\mathbb N$ the set of positive integers.

\subsection*{Notions of recurrence} The classical notion of \emph{recurrence} specializes to linear operators as follows:
\begin{definition}\label{d.rec}
	An operator $T$ acting on $X$ is called \emph{recurrent} if for every open set $U\subset X$ there exists some $k\in\mathbb N$ such that 
	\begin{align*}
		U\cap T^{-k}(U)\neq \emptyset . 
	\end{align*}
	A vector $x\in X$ is called \emph{recurrent for $T$} if there exists a strictly increasing sequence of positive integers $(k_n)_{n\in\mathbb N}$ such that 
	\begin{align*}
		T^{k_n}x\to x, 
	\end{align*}
	as $n\to+\infty$. We will denote by $\Rec(T)$ the set of recurrent vectors for $T$. 
\end{definition}
A stronger notion of (measure theoretic) recurrence, namely, \emph{measure theoretic rigidity}, has been introduced in the ergodic theoretic setting by Furstenberg and Weiss, \cite{FurWeiss}. In the context of topological dynamical systems the notions of \emph{rigidity} and \emph{uniform rigidity} have been introduced by Glasner and Maon, \cite{GM}. These notions have also been studied in linear dynamics for example in \cites{EISGRI,EIS}. The corresponding definitions are as follows.

\begin{definition}
	An operator $T$ acting on $X$ is called \emph{rigid} if there exists an increasing sequence of positive integers $(k_n)_{n\in\mathbb N}$ such that 
	\begin{align*}
		T^{k_n}x\to x \quad \text{for every} \quad x\in X 
	\end{align*}
	i.e., $T^{k_n}\to I$ (SOT) in the strong operator topology.
	
	An operator $T$ acting on $X$ is called \emph{uniformly rigid} if there exists an increasing sequence of positive integers $(k_n)_{n\in\mathbb N}$ such that 
	\begin{align*}
		\| T^{k_n}-I\| =\sup_{\| x\| \leq 1}\| T^{k_n}x-x\| \to 0. 
	\end{align*}
\end{definition}
Note that we can always assume that the sequences $(k_n)_{n\in\mathbb N}$ in the previous definitions satisfy $\lim_{n\to+\infty} k_n=+\infty$. Indeed, if $k_n$ does not converge to $+\infty$ then $T^{k_0}=I$ for some positive integer $k_0$, so that $T^{nk_0}=I$ for every positive integer $n$.

\subsection*{Notions of hypercyclicity} As we observed above, in many cases the study of the properties of recurrent operators is intimately connected to the study of hypercyclic ones. For a general overview of hypercyclicity in linear dynamics see \cite{BM} and \cite{GrossePeris}. A nice source of examples and properties of hypercyclic and supercyclic operators is the survey article \cite{GE}. See also the survey articles \cite{GE2}, \cite{MoSa2}, \cite{BoMaPe}, \cite{Fe3}, \cite{GE3}.
Here we just recall the definitions of \emph{hypercyclic} and \emph{supercyclic} operators:

\begin{definition}
	An operator $T$ acting on $X$ is called \emph{hypercyclic} if there exists $x\in X$ such that the set 
	\begin{align*}
		\Orb(x,T)\coloneqq \{T^nx: n=0,1,2,\ldots\}, 
	\end{align*}
	is dense in $X$. The set of hypercyclic vectors for $T$ is denoted by $\HC (T)$. 
	
	An operator $T$ acting on $X$ is called \emph{supercyclic} if there exists a vector $x\in X$ whose \emph{projective orbit} 
	\begin{align*}
		\mathbb C \Orb(x,T)\coloneqq\{\lambda T^nx: n=0,1,2,\ldots, \ \lambda \in\mathbb C\} 
	\end{align*}
	is dense in $X$. 
	
	An operator $T$ acting on $X$ is called \emph{cyclic} if there exists $x\in X$ such that the set 
	\begin{align*}
		\textnormal{span}\ \Orb(x,T)\coloneqq \{p(T) x:p \ \text{polynomial}\}, 
	\end{align*}
	is dense in $X$.
\end{definition}
The first trivial observations is that a hypercyclic operator is always recurrent. Indeed, every hypercyclic vector is trivially a recurrent vector and we shall see that the existence of a dense set of recurrent vectors characterizes recurrent operators. Another way to this easy conclusion is \emph{Birkhoff's transitivity theorem}, according to which, an operator $T$ acting on a separable Banach space $X$ is hypercyclic if and only if it is \emph{topologically transitive}:
\begin{definition}
	An operator $T$ acting on $X$ is called \emph{topologically transitive} if for every pair of non-empty open sets $U,V\subset X$ there exists a positive integer $n$ such that $T^n(U)\cap V \neq \emptyset$. 
\end{definition}
One can immediately compare the definition of topological transitivity above to the definition of recurrence above. Observe that there is no notion of \emph{transitivity} in the definition of a recurrent operator.

A stronger notion of hypercyclicity was introduced in \cite{BEPES}:
\begin{definition}
	An operator $T:X\rightarrow X$ is called \emph{hereditarily hypercyclic} with respect to some strictly increasing sequence of positive integers $(k_n)_{n\in\mathbb N}$ if for every subsequence $(k_{l_n})_{n\in\mathbb N}$ there exists a vector $x\in X$ such that $\overline{ \{T^{k_{l_n}}x:n\in\mathbb N\}}=X$. If $T$ is hereditarily hypercyclic with respect to the whole sequence of natural numbers we will just say that $T$ is hereditarily hypercyclic. 
\end{definition}
Of course a hereditarily hypercyclic operator is hypercyclic is a very strong sense and thus recurrent. However, it is not hard to see that a hereditarily hypercyclic operator can never be rigid. The reason is that if an operator $T$ is rigid then \emph{all} the vectors in the space are recurrent for $T$, and in fact along the same sequence of iterates of $T$.

Observe that the notions of cyclicity and hypercyclicity defined above are only meaningful when the Banach space $X$ is separable. This one other point where the theory for hypercyclic and recurrent operators becomes significantly different.

\begin{remark} The notions and definitions above where given with respect to a Banach space. However, they extend in an obvious manner to the case that $T:Y\to Y$ is a continuous linear operator acting on a \emph{Fr\'echet} space $Y$. All one needs to do is to replace the norm convergence in the definitions by convergence with respect to the metric of $Y$.
\end{remark}

The rest of this article is organized as follows. In Section \ref{s.general} we give some simple invariances of recurrent operators, motivated by similar results in hypercyclicity, and describe the spectral properties of such operators. In Sections \ref{s.power}-\ref{s.mult} we study particular classes of operators which exhibit recurrence such as power bounded operators, weighted shifts, composition operators, multiplication operators, and operators on finite dimensional spaces. In Section \ref{s.unitary} we show that the study of recurrent operators with ``sufficient structure'', acting on complex Hilbert spaces reduces to the study of recurrent unitary operators. Finally, in Section \ref{s.toplust}, we study product recurrence, again motivated by the corresponding question for hypercyclic operators. We draw some connections to orbit reflexive operators and state some relevant open problems.

\section{General properties of recurrent operators}\label{s.general} We begin the exposition by giving some easy properties of recurrent vectors and operators.

\subsection{General properties and invariances of recurrent operators} First we give an equivalent characterization of recurrence  by means of the following well known proposition; see for example \cite{Fur}. We include a proof for the sake of completeness.

\begin{proposition}
	\label{p.dense} Let $T:X\to X$ be a bounded linear operator acting on a Banach space $X$. The following are equivalent: 
	\begin{itemize}
		\item [(i)] The operator $T$ is recurrent. 
		\item [(ii)] $\overline{\Rec(T)}=X$. 
	\end{itemize}
	Furthermore, the set of recurrent vectors for $T$ is a $G_\delta$ subset of $X$. 
\end{proposition}
\begin{proof}
	In order show that (ii) implies (i) let us assume that $T$ has a dense set of recurrent points and let $U$ be an open set in $X$. Take a recurrent point $y\in U$ and $\epsilon>0$ such that $B\coloneqq B(y,\epsilon)\subset U$. Then there exists a $k\in{\mathbb N}$ such that $\| T^ky-y\|<\epsilon$. Thus $y\in U\cap T^{-k} (U)\neq \emptyset$ and so $T$ is recurrent. We now show that (i) implies (ii). To that end suppose that $T$ is recurrent and fix an open ball $B\coloneqq B(x,\epsilon)$ for some $x\in X$ and $\epsilon<1$. We need to show that there is a recurrent vector in $B$. Since $T$ is recurrent there exists a positive integer $k_1$ such that $x_1\in T^{-k_1}(B)\cap B$ for some $x_1\in X$. Since $T$ is continuous, there exists $\epsilon_1<\frac{1}{2}$ such that $B_2\coloneqq B(x_1,\epsilon_1)\subset B \cap T^{-k_1}(B)$. Now since $T$ is recurrent, there is a $k_2>k_1$ such that $x_2\in T^{-k_2}(B_2)\cap B_2$ for some $x_2\in X$. By continuity again there exists $\epsilon_2<\frac{1}{2^2}$ such that $B_3\coloneqq B(x_2,\epsilon_2)\subset B_2 \cap T^{-k_2}(B_2)$. Continuing inductively we construct a sequence $(x_n)_{n\in\mathbb N}\subset X$, a strictly increasing sequence of positive integers $(k_n)_{n\in\mathbb N}$ and a sequence of positive real numbers $\epsilon_n<\frac{1}{2^n}$, such that 
	\begin{align*}
		B(x_n,\epsilon_n)\subset B(x_{n-1},\epsilon_{n-1}), \quad T^{k_n}(B(x_n,\epsilon_n))\subset B(x_{n-1},\epsilon_{n-1}). 
	\end{align*}
	Since $X$ is complete we conclude by Cantor's theorem that 
	\begin{align*}
		\bigcap_n B(x_n,\epsilon_n)=\{y\}, 
	\end{align*}
	for some $y\in X$. It readily follows that $T^{k_n}y\rightarrow y$, that is, $y$ is a recurrent point in the original ball $B$. Finally observe that 
	\begin{align*}
		\textnormal{\Rec}(T)=\bigcap_{s=1} ^{\infty} \bigcup_{n=0} ^{\infty}\big \{ x\in X:\|T^n x -x \| <\frac{1}{s} \big\}, 
	\end{align*}
	which shows that the set of $T$-recurrent vectors is a $G_\delta$-set. 
\end{proof}

\begin{remark}
	Observe that the previous proposition remains valid whenever $T:X\to X$ is a continuous map on a complete metric space. 
\end{remark}

In the following we describe some simple invariances of the set of recurrent vectors. In particular, we show that unimodular multiples and powers of an operator share the same set of recurrent vectors. These statements are analogous to the corresponding results for hypercyclic vectors due to Ansari \cite{A1} and M\"uller and Le\'on-Saavedra \cite{MULS}.

\begin{proposition}\label{p.invariances}
	Suppose that $T$ is an operator acting on $X$. We have that 
	\begin{itemize}
		\item [(i)] For every $\lambda \in\mathbb C$ with $|\lambda|=1$ we have that $\textnormal{\Rec}(T)=\textnormal{\Rec}(\lambda T)$. 
		\item [(ii)] For every positive integer $p$ we have that $\textnormal{\Rec}(T)=\textnormal{\Rec}(T^p)$. 
	\end{itemize}
	In particular, $T$ is recurrent if and only if $T^p$ is recurrent for every positive integer $p$, if and only if $\lambda T$ is recurrent for every $\lambda \in \mathbb T$. 
\end{proposition}

\begin{proof}
	For (i) it suffices to show that $\Rec( T)\subset\Rec(\lambda T)$. Let $x\in \Rec( T) $. We define the set 
	\begin{align*}
		F \coloneqq \{ \mu\in \mathbb T: (\lambda T)^{k_n}x\to \mu x\text{ for some } (k_n)\subset \mathbb N \text{ with }k_n\to+\infty \}. 
	\end{align*}
	In order to show that $x\in \Rec(\lambda T)$ we need to show that $1\in F$.
	
	First we show that $F\neq \emptyset$. Since $x\in \Rec ( T ) $, there exists a strictly increasing sequence of positive integers $(k_n)_{n\in\mathbb N}$ such that $T^{k_n}x\to x$. By compactness, there exists a subsequence of $(k_n)_{n\in\mathbb N}$, which we call again $(k_n)_{n\in\mathbb N}$, such that $\lambda^{k_n}\to \rho$ for some $\rho\in\mathbb T$. We conclude that $(\lambda T)^{k_n} x\to \rho x$. That is $\rho\in F$.
	
	Next we show that the set $F$ is a (multiplicative) semi-group inside $\mathbb T$. Indeed, let $\mu_1,\mu_2 \in F$ and fix some $\epsilon>0$. Since $\mu_1\in F$ there is a positive integer $n_1$ such that 
	\begin{align*}
		\| (\lambda T)^{n_1}x-\mu_1 x\|<\frac{\epsilon}{2}. 
	\end{align*}
	Now since $\mu_2\in F$ there is a positive integer $n_2$ such that 
	\begin{align*}
		\| (\lambda T)^{n_2}x-\mu_2 x\|<\frac{\epsilon}{2\|(\lambda T)^{n_1}\|}. 
	\end{align*}
	We thus get 
	\begin{align*}
		\|(\lambda T)^{(n_1+n_2)}x-\mu_1 \mu_2 x\| &\leq \Norm (\lambda T)^{ n_1}\big((\lambda T)^{n_2}x -\mu_2 x\big).. +\Norm \mu_2 \big( (\lambda T)^{n_1}x-\mu_1 x\big).. \\
		&\leq \norm (\lambda T)^{n_1} .. \norm (\lambda T)^{n_2}x -\mu_2 x .. + \frac{\epsilon}{2}<\epsilon. 
	\end{align*}
	So $\mu_1\mu_2\in F$.
	
	We have already shown that there is a $\rho\in F$. Since $F$ is a semi-group this means that for every positive integer $n$, $\rho^n\in F$. If $\rho$ is a rational rotation this means that $1\in F$ and we are done. If $\rho$ is an irrational rotation there is a strictly increasing sequence of positive integers $\tau_k$ such that $\rho ^{\tau_k}\to 1$. Now we just need to observe that $F$ is closed in order to conclude that $1\in F$.
	
	For (ii) it is enough to show that $\Rec(T)\subseteq\Rec(T^p)$ since the opposite inclusion is obvious. For this, let $x\in\Rec(T)$ and take a strictly increasing sequence of positive integers $(k_n)_{n\in\mathbb N}$ such that $T^{k_n}x\to x$ as $n\to +\infty$. From this we conclude that $T^{p\ell_n+v_n}x\to x$ as $n\to +\infty$ for an increasing sequence $(\ell_n)_{n\in\mathbb N}$ and a sequence $(v_n)_{n\in\mathbb N}\subset \{0,1,2,\ldots,p-1\}$. Since $(v_n)_{n\in\mathbb N}$ is bounded we conclude that there is a $v\in\{0,1,2,\ldots,p-1\}$ such that $T^{p\ell_n+v}x \to x$ as $n\to +\infty$ for some subsequence of $(\ell_n)_{n\in\mathbb N}$ which we call again $(\ell_n)_{n\in\mathbb N}$. Let now $U$ be any open neighborhood of $x$. Since $T^{p\ell_n+v}x\to x$ there is a positive integer $m_1\coloneqq \ell_{n_1}$ such that $T^{pm_1+v}x\in U$. We have that 
	\begin{align*}
		T^{p(\ell_n+m_1)+2v}x=T^{p\ell_n+v}T^{pm_1+v}x\longrightarrow T^{pm_1+v}x\in U,\quad\text{as}\quad n\to +\infty. 
	\end{align*}
	We can thus find a positive integer $m_2\coloneqq m_1+\ell_{n_2}>m_1$ such that $T^{pm_2+2v}x\in U$. Continuing inductively we can find a positive integer $m_p\coloneqq m_{p-1}+\ell_{n_p}>m_{p-1}$ such that $T^{pm_p+pv}x\in U$. That is, $(T^p)^{m_p+v}x\in U$ which shows that $x\in \Rec(T^p)$. 
\end{proof}

\begin{remark}
	Actually, part (ii) of the previous proposition is valid whenever $X$ is a $\textbf T_1$-topological space and $T:X\to X$ is just continuous. See \cite{GOTT}. 
\end{remark}

Proposition \ref{p.invariances} has an analogue for rigid and uniformly rigid operators. In the case of uniform rigidity the proof is identical to that of Proposition \ref{p.invariances} and thus we omit it. The argument for rigid operators is slightly more subtle so we include the details of the proof.

\begin{proposition}\label{p.rigidinv} Let $T$ be an operator acting on $X$. Then,
\begin{itemize}
 \item [(i)] The operator $T$ is (uniformly) rigid if and only if, for any positive integer $p$, the operator $T^p$ is (uniformly) rigid.
 \item [(ii)] The operator $T$ is (uniformly) rigid if and only if, for any $\lambda \in\mathbb T$, the operator $\lambda T$ is (uniformly) rigid.
\end{itemize} 
\end{proposition}

\begin{proof} For (i) it is clear that $T$ is rigid whenever $T^p$ is rigid, for some positive integer $p>0$. To see the opposite implication, assume that $T$ is rigid so there exists a strictly increasing sequence of positive integers $(k_n)_{n\in\mathbb N}$ such that $T^{k_n}x\to x$ as $n\to +\infty$, for all $x\in X$. Then the uniform boundedness principle implies that $M\coloneqq \sup_{n\in\mathbb N} \|T^{k_n}\| <+\infty.$ It follows that
\begin{align*}
 \|T^{pk_n}x-x\|&\leq \|I+T^{k_n}+T^{2 k_n}+\cdots T^{(p-1)k_n}\| \|T^{k_n}x-x\|
\\
&\leq \bigg(\sum_{\ell=0} ^{p-1}M^\ell\bigg) \|T^{k_n}x-x\|.
\end{align*}
This shows that that $T^p$ is rigid whenever $T$ is rigid.

In order to show the equivalence in (ii) it suffices to show that if $T$ is rigid and $\lambda \in\mathbb T$ then $\lambda T$ is rigid. Let us define the set
\begin{align*}
 F\coloneqq \big\{\mu\in\mathbb T: \mu I\in \overline{\{\lambda T,(\lambda T)^2,(\lambda T)^3,\ldots\}}^{\operatorname{SOT}} \big\}.
\end{align*}
In order to show  that $\lambda T$ is rigid it suffices to show that $1\in F$. Since $T$ is rigid there exists a strictly increasing sequence of positive integers $(m_n)_{n\in\mathbb N}$ such that $T^{m_n}\to I$ in SOT. By compactness there exists a subsequence $(\tau_n)_{n\in\mathbb N}$ of $(m_n)_{n\in\mathbb N}$, such that $\lambda ^{\tau_n}\to \rho$, for some $\rho \in \mathbb T$, and of course we still have that $T^{\tau_n}\to I$ in SOT. Furthermore, the uniform boundedness principle implies that $M\coloneqq \sup_n \|T^{\tau_n}\|<+\infty$. We now have, for all positive integers $k$, that
\begin{align*}
 \|(\lambda T)^{k\tau_n}x-\rho^k x\| \leq \bigg(\sum_{\ell=0} ^{k-1}M^\ell\bigg) \|(\lambda T)^{\tau _n}x-\rho x\|.
\end{align*}
Observing that $(\lambda T)^{\tau_n} \to \rho I$ in SOT we conclude that $(\lambda T)^{k\tau_n}\to \rho^k I$ in SOT and thus $\rho^k \in F$ for all non-negative integers $k$. If $\rho\in\mathbb Q$ we gave that $1=\rho^{k_o}$ for some suitable positive integer $k_o$ and thus $1\in F$. If $\rho \in \mathbb R\setminus \mathbb Q$ then there is a sequence $(m_n)_{n\in\mathbb N}$ such that $\rho^{m_n}\to 1$ as $n\to +\infty$. Since $(\rho^{m_n})_{n\in\mathbb N} \subset F$ and $F$ is closed we conclude that $1\in F$.
\end{proof}

\begin{proposition}
	\label{p.invertible} Let $T:X\rightarrow X$ be an invertible operator. Then $T$ is recurrent if and only if $T^{-1}$ is recurrent. 
\end{proposition}

\begin{proof}
	It suffices to show that if $T$ is recurrent then $T^{-1}$ is recurrent. Since 
	\begin{align*}
		\Rec(T^{-1})=\bigcap_{s=1}^{\infty} \bigcup_{n=0}^{\infty} \bigg\{ x\in X: \| T^{-n}x-x\| <\frac{1}{s} \bigg\}, 
	\end{align*}
	in view of Baire's category theorem it suffices to show that, given $s\in \{ 1,2,\ldots \}$, $\epsilon >0$ and $y\in X$, there exist $x\in X$ and $n\in \{ 0,1,2,\ldots\}$ such that $\| y-x\| <\epsilon$ and $\| T^{-n}x-x\| <1/s$. Indeed, we can choose a $z\in \Rec(T)$ such that $\| y-z\| <\epsilon /2$. We also may find a positive integer $n$ such that $\| z-T^nz\| <\min \{1/s, \epsilon/2 \}$. Define $x=T^nz$. Then we have $\| y-x\| \leq \| y-z\| +\| z-x\| <\epsilon $ and 
	$\| T^{-n}x-x\| =\| z-T^n z\| <1/s$. This completes the proof of the proposition. 
\end{proof}

\begin{remark}
	If $T$ is invertible, the operators $T,T^{-1}$ do not necessarily share the same recurrent vectors. Consider for example a hypercyclic, invertible, bilateral backward weighted shift $B_w$, on $\ell^2(\mathbb Z)$. For a detailed definition see \S~\ref{s.weighted}. Such hypercyclic weighted shifts exist as shown in \cite{Salas}. Since the hypercyclic vectors of $B_w$ are dense in $\ell^2$ there exists a $x=(x_n)_{n\in\mathbb Z} \in\ell^2(\mathbb Z)$ which is hypercyclic for $B_w$ and satisfies $x_0\neq 0$. Let $y=(y_n)_{n\in\mathbb Z}\in\ell^2(\mathbb Z)$ be the vector with $y_n\coloneqq x_n$ if $n\geq 0$ and $y_n\coloneqq 0$ if $n<0$. Observe that $(B_w ^k y)_n=(B_w ^k x)_n$ if $n\geq -k$. We claim that $y\in\HC(B_w)$. To see this, let $\epsilon >0$ and $z\in\ell^2(\mathbb Z)$. There exists some $n_0$ such that $\sum_{|n|> n_0}|z_n|^2 <\epsilon^2/4$. Now let $N>n_o$ such that $\|B_w ^N x-z\|_2<\epsilon/2$. We can estimate
	\begin{align*}
		\|B^N _w y-z\|_2& \leq \bigg(\sum_{n\leq -N} |(B^N _wy)_n-z_n|^2\bigg)^\frac{1}{2}+\bigg(\sum_{n\geq - N} |(B^N _wy)_n-z_n|^2\bigg)^\frac{1}{2}
		\\
		&\leq  \bigg(\sum_{|n|\geq N} |z_n|^2\bigg)^\frac{1}{2}+\|B^N _w x-z\|_2<\epsilon.
	\end{align*}
This shows the claim and thus $y\in\HC(B_w)\subseteq \Rec(B_w)$. On the other hand, $B_w ^{-1}$ is a bilateral forward weighted shift thus $\|B_w ^{-n} y-y\|_2\geq |y_0|\neq 0 $ for every $n\geq 1$. This shows that $y$ cannot be recurrent for $B_w ^{-1}$.	
\end{remark}

The following lemma is an immediate consequence of the definitions and states that the notions of recurrence, rigidity and uniform rigidity are invariant under similarity transformations. We omit the simple proof.

\begin{lemma}
	\label{l.equiv} Let $T:X\to X$ be an operator acting on $X$ and $S:X\to X$ be an invertible bounded operator. Then: 
	\begin{itemize}
		\item[(i)] $T$ is recurrent if and only if $S^{-1} T S$ is recurrent. 
		\item[(ii)] $T$ is rigid if and only if $S^{-1} T S$ is rigid. 
		\item[(iii)] $T$ is uniformly rigid if and only if $S^{-1} T S$ is uniformly rigid. 
	\end{itemize}
\end{lemma}

\subsection{Spectral properties of recurrent operators.} In this paragraph we study some spectral properties of recurrent operators. We denote by $\sigma(T)$ the \emph{spectrum} of $T$, $\sigma_p(T)$ the \emph{point-spectrum} of $T$ and by $r(T)$ the \emph{spectral radius} of $T$. In general we will see that the spectral properties of recurrent operators tend to resemble the spectral properties of hypercyclic operators.

\begin{proposition}
	\label{p.spectralradius} Let $T:X\rightarrow X$ be an operator acting on $X$. If $r(T)<1$ then $T$ is not recurrent. 
\end{proposition}

\begin{proof}
	Since $r(T)<1$, we have that $\| T^n\| \to 0$ as $n\to +\infty$. Hence $\| T^n x\| \to 0$ for every $x\in X$ and therefore $T$ is not recurrent. 
\end{proof}

\begin{proposition}
	\label{p.decomp} Let $T_1$, $T_2$ be two operators acting on the Banach spaces $X_1$, $X_2$ respectively. If $T_1\oplus T_2$ is recurrent then both $T_1$, $T_2$ are recurrent operators. 
\end{proposition}

\begin{proof}
	Take a recurrent vector $x_1\oplus x_2$ for $T_1\oplus T_2$. It is clear that $x_1$, $x_2$ are recurrent vectors for $T_1$, $T_2$ respectively. The last implies that $T_1$, $T_2$ are recurrent operators. 
\end{proof}

\begin{proposition}
	\label{p.spectrum} If $T:X\to X$ is a recurrent operator then every component of the spectrum of $T$, $\sigma (T)$, intersects the unit circle $ \mathbb T$. 
\end{proposition}

\begin{proof}
	We first show the following statement 
	\begin{align}
		\label{e.claim} \text{if}\quad \sigma (T)\subset \{ \lambda : |\lambda | >1 \}\quad \text{then}\ T \ \text{is not recurrent.} 
	\end{align}
	Indeed, the hypothesis on the spectrum implies that $T$ is invertible and that $r(T^{-1})<1$. By Proposition \ref{p.spectralradius} it follows that $T^{-1}$ is not recurrent and by Proposition \ref{p.invertible} we conclude that $T$ is not recurrent.
	
	Let us now prove the full conclusion of the proposition. Arguing by contradiction, assume that some component $C_1$ of the spectrum $\sigma (T)$ does not intersect the unit circle. Then either $C_1\subset \mathbb D$ or $C_1\subset \mathbb C \setminus \overline{\mathbb D}$. By \cite{BM}*{Lemma 1.21} there exists a clopen set $\sigma_1 \subset \sigma (T)$ such that either $C_1\subset \sigma_1 \subset \mathbb D$ or $C_1\subset \sigma_1\subset \mathbb C \setminus \overline{\mathbb D}$. By the Riesz decomposition theorem applied for $\sigma_1$ and $\sigma_2:=\sigma (T) \setminus \sigma_1$ there exist operators $T_1$, $T_2$ and a decomposition of the space $X$, $X=X_1\oplus X_2$ such that $T=T_1\oplus T_2$, $T_i:X_i\to X_i$, $i=1,2$ and $\sigma (T_i)=\sigma_i$, $i=1,2$. Now, Proposition \ref{p.spectralradius} and \eqref{e.claim} imply that $T_1$ is not recurrent and by Proposition \ref{p.decomp} we conclude that $T$ is not recurrent, thus reaching a contradiction. This completes the proof of the proposition. 
\end{proof}

\begin{corollary}
	A compact operator on an infinite dimensional Banach space cannot be recurrent. 
\end{corollary}

\begin{example}
	\label{ex.compact} It is well known that there exist compact operators $K$ acting on separable Banach spaces $X$ such that $I+K$ is hypercyclic and thus recurrent. It is however not difficult to construct a compact operator $K$ such that $I+K$ is recurrent but not hypercyclic. Indeed, consider the space $\ell^2(\mathbb N)$ and take $K$ to be the operator 
	\begin{align*}
		K(x_1,x_2,\ldots,x_j,\ldots)\coloneqq ((e^{i\theta}-1) x_1,0,\ldots,0,\ldots), 
	\end{align*}
	for some $\theta\in \mathbb R$. Then for every positive integer $n$ we have that 
	\begin{align*}
		(I+K)^n x -x= (e^{in\theta}x_1,x_2,\ldots,x_j,\ldots)-x=( (e^{in \theta}-1)x_1,0,\ldots,0,\ldots). 
	\end{align*}
	Thus $K$ is a compact operator with one-dimensional range and $I+K$ is recurrent. An obvious modification provides a compact operator $K$ with $d$-dimensional range, for any positive integer $d$, such that $I+K$ is recurrent. Of course $I+K$ cannot be hypercyclic in this case.
	
	In a similar fashion one can construct a compact operator $K$ with infinite dimensional range such that $I+K$ is recurrent but not hypercyclic. For this just take a sequence $(\theta_n)_{n\in\mathbb N}$ with $\theta_n\to 0$ as $n\to +\infty$ and define $K:\ell^2(\mathbb N)\to \ell^2(\mathbb N)$ as 
	\begin{align*}
		K(x_1,x_2,\ldots,x_j,\ldots)=((e^{i\theta_1}-1)x_1,(e^{i\theta_2}-1)x_2,\ldots,(e^{i\theta_j}-1)x_j,\ldots). 
	\end{align*}
	Again we have for every positive integer $n$ 
	\begin{align*}
		(I+K)^nx=(e^{in \theta_1}x_1,e^{in \theta_2}x_2,\ldots,e^{in \theta_j}x_j,\ldots). 
	\end{align*}
	Based on the previous identity we can show that $I+K$ is recurrent. Indeed, note that for every positive integer $m$ and every $a=(a_1,\ldots,a_m)\in \mathbb T^m$ there exists a strictly increasing sequence of positive integers $(k_n)_{n\in\mathbb N}$ such that $(a_1 ^{k_n},a_2 ^{k_n},\ldots,a_m ^{k_n})\to (1,1,\ldots,1)$ as $n\to +\infty$. Using the fact that $\theta_n\to 0$ we readily see that $K$ is compact. 
\end{example}

\begin{proposition}
	\label{p.pspectrum} Let $T:X\to X$ be an operator. If $T$ is recurrent then for every $\lambda \in \mathbb{C}\setminus \mathbb T$ the operator $T-\lambda I$ has dense range, hence $\sigma_p(T^*)\subset \mathbb T$. Here $T^*$ denotes the Banach space adjoint of the operator $T$. 
\end{proposition}

\begin{proof}
	Suppose that $\overline{(T-\lambda I)(X)}\neq X$ for some $\lambda \in \mathbb{C}\setminus \mathbb T$. Since $T$ is recurrent and the set $X\setminus \overline{(T-\lambda I)(X)}$ is non-empty and open there exists a non-zero vector $x\in X$ such that $x\in \Rec(T)\cap X\setminus \overline{(T-\lambda I)(X)}$. By the Hahn-Banach theorem there exists $x^*\in X^*$ such that $x^*(x)\neq 0$ and $x^*(\overline{(T-\lambda I)(X)})=\{ 0\}$. Then for every $y\in X$ we have $x^*(Ty)=\lambda x^*(y)$ and thus $x^*(T^ny)={\lambda}^nx^*(y)$ for every $n=1,2,\ldots $. Since $x\in \Rec(T)$ there exists a sequence of positive integers $(k_n)_{n\in\mathbb N}$ such that $k_n\to +\infty $ and $T^{k_n}x\to x$. Hence ${\lambda }^{k_n}x^*(x)=x^*(T^{k_n}x)\to x^*(x)$. Using that $x^*(x)\neq 0$ we conclude that ${\lambda }^{k_n}\to 1$, which is a contradiction since $\lambda \in \mathbb{C}\setminus \mathbb T$. This completes the proof. 
\end{proof}

\begin{remark}
	If $T$ is hypercyclic then $T$ is recurrent but $\sigma_p(T^*)=\emptyset$. However, there exist several recurrent operators such that $\emptyset \neq \sigma_p(T^*)\subset \mathbb T$. For example, this is the case for the operator $I+K$ constructed in Example \ref{ex.compact} as well as for unimodular multiples of the identity operator. 
\end{remark}

The following lemma contains the classical fact that a sufficiently large supply of eigenvectors corresponding to unimodular eigenvalues implies that the operator is recurrent. 

\begin{lemma}
	\label{l.suffrecurrent} Let $T:X\to X$ be an operator. If $T$ has \emph{discrete spectrum}, that is, if 
	\begin{align}
		\label{e.unimod} \overline{ \textnormal{span} \{ x\in X:Tx=\lambda x \ \text{for some}\ \lambda \in {\mathbb T} \} }=X 
	\end{align}
	then $T$ is recurrent. 
\end{lemma}

\begin{proof}
	Take $x,y\in X$ such that $Tx=\mu_1 x$, $Ty=\mu_2 y$ for some $\mu_1,\mu_2 \in {\mathbb T}$ and fix any $\lambda_1 ,\lambda_2\in {\mathbb C}$. Since $\mu_1 ^{k_n}\to 1$ and $\mu_2 ^{k_n}\to 1$ for some strictly increasing sequence of positive integers $(k_n)_{n\in\mathbb N}	$ we get that $\lambda_1x+\lambda_2y\in \Rec(T)$. The last implies that 
	\begin{align*}
		\textnormal{span}\{ x\in X: Tx=\lambda x \ \text{for some} \ \lambda \in {\mathbb T} \} \subset \Rec(T) 
	\end{align*}
	and by our hypothesis we conclude that $T$ is recurrent. 
\end{proof}

\begin{remark}A few remarks are in order.
 
\begin{itemize}
  \item [(i)] Observe that Lemma \ref{l.suffrecurrent} holds even in the case that the space $X$ is non-separable. A much stronger form of the hypothesis \eqref{e.unimod} is the assumption that an operator $T$, acting on \emph{separable} Banach space $X$, has a perfectly spanning set of eigenvectors associated to unimodular eigenvalues. This means that there exists a continuous probability measure $\sigma$ on $\mathbb T$ such that, for every Borel $A\subset \mathbb T$ with $\sigma(A)=1$ we have 
\begin{align*}
	\overline{ \textnormal{span} \{ x\in X:Tx=\lambda x \ \text{for some}\ \lambda \in A \}}=X. 
\end{align*}
In this case we get the stronger conclusion that $T$ is hypercyclic. In fact, in this case, $T$ is \emph{frequently hypercyclic}; see \cites{Griv,Griv1,BaGr2,BaGr3}. The point of Lemma \ref{l.suffrecurrent} is that if we only assume the weaker hypothesis \eqref{e.unimod} we can still conclude that $T$ is recurrent.

\item[(ii)] On a similar spirit, if besides \eqref{e.unimod} we further assume that $T:X\to X$ is power bounded, i.e. $\sup_{n\in\mathbb N}\|T^n\|<+\infty$, and $X$ is \emph{separable}, then $T$ is rigid; see \cite{EIS}. However, \eqref{e.unimod} is far from being a necessary condition for rigidity since there exist rigid unitary operators whose point spectrum is empty and so \eqref{e.unimod} fails for these operators. For such examples see \cite{EIS} and the references therein.

\item[(iii)] The assumption \eqref{e.unimod} alone does not suffice in order to conclude that $T$ is rigid. To see this consider any hereditarily hypercyclic operator which has a dense set of \emph{periodic points}, namely, points $x\in X$ for which there exists a positive integer $n$ with $T^nx=x$. One such example is provided by the operator $\lambda B$ on the space of square summable sequences, where $B$ is the unweighted unilateral backward shift and $|\lambda|>1$. Observe that the operator above is hypercyclic and has a dense set of periodic points, that is, it is \emph{chaotic}.

\item[(iv)] In the case of complex separable Hilbert spaces $H$, condition \eqref{e.unimod} appears in \cite{Flytz}, and is shown to be equivalent to the existence of an invariant Borel probability measure of square integrable norm. For the precise definitions see \cite{Flytz}. We only note here that this class of Borel probability measures contains the interesting class of Gaussian measures.   
\end{itemize}
\end{remark}

\subsection{Spectral properties of rigid and uniformly rigid operators} We already saw that every component of the spectrum of a recurrent operator meets the unit circle. If an operator is rigid then we also get that the spectrum must be contained in the closed unit disk of the complex plane.
\begin{proposition}\label{p.rigspec} Let $T$ be a rigid operator acting on a Banach space $X$. Then every component of the spectrum of $T$ intersects the unit circle. Furthermore we have that $\sigma(T)\subseteq \overline{\mathbb D}$.
\end{proposition}

\begin{proof} If $T$ is rigid then it is recurrent so Proposition \ref{p.spectrum} gives the first assertion of the proposition. We also claim that $r(T)=1$. Indeed, if $r(T)>1$ then follows by \cite{Muller}*{Corollary 1.2} that there exists a non-zero vector $y\in X$ such that $\|T^n y\|\to +\infty$ as $n\to +\infty$. On the other hand, since $T$ is rigid there exists a strictly increasing sequence of positive integers $(k_n)_{n\in\mathbb N}$ such that $T^{k_n}x\to x$ for every $x\in X$. Thus we  should also have that $\|T^{k_n}y\|\to \|y\|$, a contradiction.
\end{proof}

For uniformly rigid operators we have a significant strengthening of the previous statement.

\begin{proposition} Let $T$ be a uniformly rigid operator acting on a Banach space $X$. Then the spectrum of $T$ is contained in the unit circle. In particular, if $T$ is uniformly rigid then $T$ is invertible.
\end{proposition}

\begin{proof} Let $T$ be a uniformly rigid operator and suppose that $(k_n)_{n\in\mathbb N}$ is a strictly increasing sequence of positive integers such that $\|T^{k_n}-I\|\to 0$ as $n\to+\infty$.  Without loss of generality we can assume that $k_n\to +\infty$.  In particular we have that $\sup_{n\in\mathbb N} \|T^{k_n}\|<+\infty$. By Proposition \ref{p.rigspec} we have that $\sigma(T)\subset \overline{\mathbb D}$. Since $T$ is uniformly rigid it is immediate that $\sigma_p(T)\cap \mathbb D=\emptyset$ and by Proposition \ref{p.pspectrum} we also have that $\sigma_p(T^*)\cap \mathbb D=\emptyset$. Thus if $\lambda \in \sigma(T)\cap \mathbb D$ then $\lambda$ is necessarily in the approximate point spectrum of $T$.

Let $\lambda \in \sigma(T)\cap \mathbb D$. By Proposition \ref{p.rigidinv} we have that $T^p$ is uniformly rigid for any positive integer $p$. By the spectral theorem $\lambda^p \in \sigma(T^p)$ and by the previous discussion $\lambda^p$ is necessarily an approximate eigenvalue of $T^p$. This means that, for every positive integer $p$, there exists a sequence $(x_n ^{(p)})_{n\in\mathbb N}$ with $\| x_n ^{(p)}\|=1$ such that 
$\|T^p x_n ^{(p)}-\lambda^p x_n ^{(p)}\|\to 0$ as $n\to+\infty$. Using this we can construct a sequence $(y_n)_{n\in\mathbb N}\subset X$ with $\|y_n\|=1$ for every $n\in\mathbb N$, such that $\|T^n y_n -\lambda^ny_n\| <\frac{1}{n}$ for every integer $n\geq 1$. This immediately implies that
\begin{align*}
\|T^n y_n\|\leq \frac{1}{n}+|\lambda|^n \|y_n\|\leq \frac{1}{n}+|\lambda|^n  \to 0 \quad \text{as} \quad n\to +\infty.
\end{align*}
On the other hand, since $T$ is uniformly rigid along the sequence $(k_n)_{n\in\mathbb N}$ we have that
\begin{align*}
 \Abs{\|T^{k_n} y_{k_n}\|-1}	 \leq \|T^{k_n}y_{k_n}-y_{k_n}\|\leq \|T^{k_n}-I\|  \to 0\quad \text{as}\quad n\to+\infty,
\end{align*}
which is clearly a contradiction. Thus $\sigma(T)\subseteq \mathbb T$.
\end{proof}

\begin{question} We have already seen that if an operator $T$ is invertible then $T$ is recurrent if and only if $T^{-1}$ is recurrent. On the other hand, for uniformly rigid operators we get that $T$ is automatically invertible. However, it is not clear whether $T^{-1}$ is also uniformly rigid without some additional information on $T$. See also Proposition \ref{p.powerigid}. It is natural to ask if the same property is shared by rigid operators, namely, whether every rigid operator is invertible. Failing that, is it true that if $T$ is rigid and invertible then $T^{-1}$ is also rigid? Note that both questions above have affirmative answers in all the examples of rigid and uniformly rigid operators appearing in this paper.
\end{question}

\section{Power bounded operators}\label{s.power} Recall that an operator $T:X\to X$ is called power bounded provided there exists a positive number $M$ such that $\| T^n \| \leq M$ for every positive integer $n$. The main purpose of this section is to show that power bounded recurrent operators are similar to surjective isometries. This is contained in Proposition \ref{p.bimpliesunit} below. In view of Lemma \ref{l.equiv} this means that the study of power bounded operator recurrent operators reduces to the study of recurrent surjective isometries.

We start with a simple lemma. 

\begin{lemma}
	\label{l.rec-closed} If $T:X\to X$ is a power bounded operator then the set $\Rec(T)$ is closed. 
\end{lemma}

\begin{proof}
	Let $(x_n)_{n\in\mathbb N}\subset \Rec(T)$ and $x\in X$ and suppose that $x_n\to x$ in $X$. Since $(x_n)_{n\in\mathbb N}\subset \Rec(T)$ we can choose a strictly increasing subsequence of positive integers $(m_n)_{n\in\mathbb N}$ such that $\lim_{n\to +\infty}\|T^{m_n}x_n-x_n\|=0$. Using the hypothesis that $T$ is power bounded it is routine to check that $T^{m_n}x\to x$ in $X$ thus $x\in \Rec(T)$. This shows that $\Rec(T)$ is a closed set. 
\end{proof}
We continue with our main result for this section.
\begin{proposition}
	\label{p.bimpliesunit} Let $T:X\to X$ be an operator. 
	\begin{itemize}
		\item[(i)] If $\| T\| \leq 1$ and $T$ is recurrent then $T$ is a surjective isometry. 
		\item[(ii)] If $T$ is power bounded and recurrent then $\sigma (T)\subset\mathbb T$ and $T^{-1}$ is power bounded and recurrent. In particular, $T$ is similar to an invertible isometry. 
		\item[(iii)] If $T$ is power bounded, recurrent and $\sigma (T)\cap \mathbb T= \{ \lambda\}$ for some $\lambda \in \mathbb T$ then $T=\lambda\, I$. 
		\item[(iv)] If $T$ is power bounded and recurrent then $T$ is not supercyclic. 
		\item[(v)] If $T$ is power bounded and recurrent then $\Rec(T)=X$. 
	\end{itemize}
\end{proposition}

\begin{proof}
	Assertion (v) follows from Lemma \ref{l.rec-closed} and the fact that $T$ is recurrent. In order to show (i) we fix some $x\in X$. By (v) we have that $x\in\Rec(T)$ thus there exists a strictly increasing sequence of positive integers $(k_n)_{n\in\mathbb N}$ such that $T^{k_n}x\to x$. The assumption that $T$ is a contraction implies that $\| T^{n+1}x\|\leq \| T^nx\|$ for every $n$. Since $\| T^{k_n}x\| \to \| x\|$ it follows that $\| T^nx\| \to \| x\|$. From the last we conclude that $\| Tx\|=\| x\|$. It remains to show that $T$ is surjective. Observe that it is enough to prove the convergence of the sequence $(T^{k_n-1}x)_{n\in\mathbb N}$. We have $\| T^{k_n-1}x - T^{k_l-1}x\|=\| T^{k_n}x - T^{k_l}x\|\to 0$ as $n,l\to +\infty$. Therefore $( T^{k_n-1}x)_{n\in\mathbb N}$ is a Cauchy sequence. This completes the proof of (i).
	
	We proceed with the proof of (ii). Define the equivalent norm $\| x \|_1 =\sup_{n\geq 0} \| T^nx\|$, $x\in X$. Then $T$ is a contraction and recurrent operator on the Banach space $(X,\| \cdot \|_1)$. By (i) it follows that $T$ is a surjective isometry on $(X,\| \cdot \|_1)$, hence $\sigma (T)\subset \mathbb T$. Therefore $T$ is invertible on $(X,\| \cdot \|_1)$. It is clear now that $T$ is also invertible on $(X, \| \cdot \| )$. We have $ \| T^{-n}x \| \leq {\| T^{-n}x\|}_1=\| x\|_1$ for every $ n\geq 0 $ and every $x\in X$. Thus $T^{-1}: (X,\| \cdot \| ) \to (X,\| \cdot \| )$ is power bounded and by \cite{DriMbe}*{Lemma 9} we conclude that $T: (X,\| \cdot \| ) \to (X,\| \cdot \| )$ is similar to an invertible isometry. The proof of (ii) is complete.
	
	Let us now prove (iii). We have $\sigma (\frac{T}{\lambda})\cap \mathbb T= \{ 1\}$. The theorem of Katznelson and Tzafriri \cite{KaTza}, gives 
	\begin{align*}
		\lim_{n\to\infty} \NORm \frac{T^{n+1}}{\lambda^{n+1}}-\frac{T^n}{\lambda^n} ..=0, 
	\end{align*}
	or equivalently 
	\begin{align*}
		\lim_{n\to\infty} \| T^n(T-\lambda\, I) \| =0. 
	\end{align*}
	Take a non-zero vector $x\in X$. There exists a strictly increasing sequence of positive integers $(k_n)_{n\in\mathbb N}$  such that $T^{k_n}x\to x$, hence $T^{k_n}(T-\lambda\, I)x\to (T-\lambda\, I)x$. It is now clear that $Tx=\lambda x$.
	
	 Ansari and Bourdon have showed in \cite{AB} that, if $T$ is power bounded and supercyclic, then it is stable, that is, $\| T^n x \| \to 0$ for every $x\in X$. Assertion (iv) follows.
	\end{proof}

\begin{remark} As it is observed in \cite{EIS}*{Remark 2.2}, rigid contractions on Hilbert spaces are necessarily unitary. Here we show the stronger statement that \emph{recurrent contractions} on Banach spaces are surjective isometries and, more generally, recurrent power bounded operators on complex Banach spaces are similar to invertible isometries.
\end{remark}

We close this section by an easy remark on rigid and uniformly rigid power bounded operators.

\begin{proposition}\label{p.powerigid} Let $T$ be an operator acting on a Banach space $X$. Then $T$ is power bounded and (uniformly) rigid if and only if $T^{-1}$ is power bounded and (uniformly) rigid.
\end{proposition}

\begin{proof} We will just show the proposition for rigid operators, the proof for the case of uniform rigidity being a repetition of the same arguments. So assume that $T$ is power bounded and rigid. Then Proposition \ref{p.bimpliesunit}, (ii), implies that $T$ is invertible and $T^{-1}$ is power bounded. It remains to show that $T^{-1}$ is rigid. Since $T$ is rigid there exists a strictly increasing sequence of positive integers $(k_n)_{n\in\mathbb N}$ such that $T^{k_n}x\to x$ for all $x\in X$. Thus, for every $x\in X$ we have 
\begin{align*}
 \|T^{-k_n}x-x\|\leq \|T^{-k_n}\| \|x-T^{k_n}x\|\leq \sup_{m\in\mathbb N} \|T^{-m}\| \|x-T^{k_n}x\|,
\end{align*}
which shows that $T^{-1}$ is rigid with the same sequence $(k_n)_{n\in\mathbb N}$.
\end{proof}

\section{Finite dimensional spaces} In this section we include a characterization of the recurrent operators $T:\C^d\to \C^d$ and $T:\R^d\to \R^d$. This is relatively straightforward and probably well known. However we provide the details here adjusted to our terminology.

We begin with the complex case. 

\begin{theorem}
	\label{t.complex} A matrix $T:\C^d\to \C^d$ is recurrent if and only if it is similar to a diagonal matrix with unimodular entries.
\end{theorem}

\begin{proof}
	We first assume that $T:\C^d\to \C^d$ is recurrent. Since $\sigma_p(T)=\sigma(T)$ and by Proposition \ref{p.spectrum} every component of the spectrum of $T$ intersects the unit circle, we conclude that $\sigma(T)=\sigma_p(T)=\{\lambda_1,\ldots,\lambda_M\}$ for some $\lambda_1,\ldots,\lambda_M\in\mathbb T$, with multiplicities $m_1,\ldots,m_M$, respectively, and $m_1+\cdots+m_M=d$. By the canonical Jordan decomposition the matrix $T$ is similar to a block-diagonal matrix $\tilde T$ of the form 
	\begin{align*}
		\tilde T\coloneqq 
		\begin{pmatrix}
			V_1& 0 & \cdots &0 \\
			0 & V_2 & \cdots & 0 \\
			\vdots & \vdots & \ddots & \vdots \\
			0 & 0 & \cdots & V_M\\
		\end{pmatrix}
		, 
	\end{align*}
	where each $V_j$, $j=1,\ldots,M$, is either a $m_j\times m_j$ Jordan matrix, that is, 
	\begin{align}
		\label{e.jordan} V_j= 
		\begin{pmatrix}
			\lambda_j &1 & \cdots&0 &0 \\
			0 & \lambda_j & \ddots &0 & 0 \\
			\vdots & \vdots & \ddots & 1 &\vdots \\
			0 & 0 & \cdots & \lambda_j & 1 \\
			0 & 0 & \cdots & 0 & \lambda j\\
		\end{pmatrix}
		, 
	\end{align}
	or of the form $V_j=\lambda_j I_{m_j}$, where $I_{m_j}$ is the $m_j$-dimensional identity matrix.
	
	We claim that each block $V_j$ is of the form $\lambda_j I_{m_j}$. Arguing by contradiction we assume that there exists at least one Jordan block $V_{j_o}$ of the form \eqref{e.jordan}, with $m_{j_o}\geq 2$. This implies that the $2\times 2$ matrix $V$, where 
	\begin{align*}
		V= 
		\begin{pmatrix}
			\lambda_{j_o} &1 \\
			0 & \lambda_{j_o}\\
		\end{pmatrix}
		, 
	\end{align*}
	is recurrent on $\mathbb C^2$. An easy calculation shows that for every natural number $n$ we have 
	\begin{align*}
		V^n = 
		\begin{pmatrix}
			\lambda_{j_o} ^n & n\lambda_{j_o}^{n-1} \\
			0 & \lambda_{j_o} ^n\\
		\end{pmatrix}
		. 
	\end{align*}
	Since $\overline{\Rec(V)}=\C^2$, there exists a recurrent vector $z=(z_1,z_2)^t \in\C^{2\times 1}$ with $z_2\neq 0$. Hence, there exists a strictly increasing sequence of positive integers $(k_n)_{n\in\mathbb N}$ such that $V^{k_n}z\to z$. Since $z_2\neq 0$ this implies that $\lambda_{j_o} ^{k_n}\to 1$. On the other hand we must have 
	\begin{align*}
		\lambda_{j_o} ^{k_n} z_1+k_n\lambda_{j_o} ^{k_n-1}z_2\to z_1. 
	\end{align*}
	This is clearly impossible since $z_2\neq 0$ and this contradiction proves the claim.
	
	We have showed that $T$ is similar to a diagonal matrix with unimodular entries. The opposite direction follows by observing that if $(a_1,\ldots,a_d)\in \mathbb T^d$ then there exists a strictly increasing sequence of positive integers $(k_n)_{n\in\mathbb N}$ such that $a_j ^{k_n}\to 1 $ for all $j=1,2,\ldots,d$. 
\end{proof}

We now move to the study of the real case.

\begin{theorem}
	A matrix $T:\R^d\to \R^d$ is recurrent if and only if it is similar to a block diagonal matrix if the form
	\begin{align*}
		\begin{pmatrix}
			J_1 &0 & \cdots&0 &0 \\
			0 & J_2 & \ddots &0 & 0 \\
			\vdots & \vdots & \ddots & 0 &\vdots \\
			0 & 0 & \cdots &J_{M-1} & 0 \\
			0 & 0 & \cdots & 0 & J_M\\
		\end{pmatrix}
		, 
	\end{align*}
	where each $J_j$, $1\leq j \leq M,$ is either a $2\times 2$ rotation matrix or a $1\times 1$ matrix with entry either $1$ or $-1$. 
\end{theorem}

\begin{proof}
	Let $T:\R^d\to \R^d$ be recurrent. By the canonical Jordan decomposition $T$ is similar to a block-diagonal matrix consisting of blocks $J_1,\ldots J_M$. Each block $J_j$ is a real Jordan block and for real Jordan blocks there are two mutually exclusive cases:
\paragraph{\textbf{case 1:}} The Jordan block $J_j$ is identical to a complex Jordan block with real eigenvalue. In this case we conclude by the same argument as in the proof of Theorem~\ref{t.complex} that $J$ is of the form $J_j=\lambda_j I_{m_j}$ with $\lambda_j \in\{-1,+1\}$ and $1\leq m_j \leq d$. 
\paragraph{\textbf{case 2:}} The Jordan block $J_j$ is a block matrix of the form
	\begin{align}\label{e.jordanreal}
		 J= 
		\begin{pmatrix}
			C &I_2 & \cdots&0 &0 \\
			0 & C & \ddots &0 & 0 \\
			\vdots & \vdots & \ddots & I_2 &\vdots \\
			0 & 0 & \cdots &C & I_2 \\
			0 & 0 & \cdots & 0 & C\\
		\end{pmatrix}
		, 
	\end{align}
	where $C$ is a $2\times 2$ matrix
	\begin{align*}
		C = 
		\begin{pmatrix}
			a & b \\
			-b & a\\
		\end{pmatrix}
		,
	\end{align*}
with $a,b\in\R$.
	
	Since $T$ is recurrent and $T$ is a block diagonal matrix consisting of the blocks $J_j$, Lemma \ref{l.equiv} implies that each block $J_j$ is itself recurrent on the corresponding subspace of $\R^d$. Likewise, since the last block row of $J_j$ is orthogonal to all but the last block columns of $J_j$, the $2\times 2$ matrix $C$ has to be recurrent on $\R^2$. However, for every $x=(x_1,x_2) \in\R^2$ and every positive integer $n$ we have 
	\begin{align*}
		\|C^n x\|=(a^2+b^2)^\frac{n}{2}\|x\|, 
	\end{align*}
	where $\|\cdot\|$ is the Euclidean norm on $\R^2$. From this identity it readily follows that $C$ is recurrent if and only if $a^2+b^2=1$. Thus $C$ is a rotation. 
	
	We now show that if $T$ is recurrent then every $J$ is itself a rotation $C$. Indeed, suppose that $J$ has at least two blocks. Then the block matrix 
	\begin{align*}
		S = 
		\begin{pmatrix}
			C & I_2 \\
			O & C\\
		\end{pmatrix}
	\end{align*}
	must be recurrent on $\R^4$, for some rotation matrix $C$. However, the iterates of $S$ have the form 
	\begin{align*}
		S^n = 
		\begin{pmatrix}
			C^n & nC^{n-1} \\
			O & C^n\\
		\end{pmatrix}
		. 
	\end{align*}
	By an argument identical to the one used in the case of a complex Jordan matrix this leads to a contradiction.
	
	The considerations above show that if $T$ is recurrent then $T$ is similar to a block diagonal matrix, with each block being either a rotation or a $1\times 1$ matrix with entry either $1$ or $-1$. Conversely, every matrix of this form is easily seen to be recurrent. 
\end{proof}

\begin{remark} The results in this section show that, in finite dimensions, recurrence of linear operators is equivalent to uniform rigidity.
 
\end{remark}

\section{Weighted shifts and diagonal operators}\label{s.weighted} In this section we study the recurrence properties of weighted shifts and diagonal operators on classical sequence spaces. We will denote by $\ell^p(\mathbb Z)$ the Banach space of doubly indexed sequences $a=(a_n)_{n\in\mathbb Z}$ such that 
\begin{align*}
	\|a\|_p\coloneqq \big( \sum_{n\in\mathbb Z} |a_n|^p\big)^\frac{1}{p}<+\infty,\quad 1\leq p <\infty, 
\end{align*}
and 
\begin{align*}
	\|a\|_\infty \coloneqq \sup_{n\in\mathbb Z} |a_n|<+\infty. 
\end{align*}
Let $w=(w_n)_{n\in\mathbb Z}$ be a \emph{weight sequence}, that is a bounded sequence of positive real numbers. The \emph{bilateral weighted backward shift} with \emph{weight sequence} $w$ is the linear operator $B_w:\ell^p(\mathbb Z)\to \ell^p(\mathbb Z)$ defined as 
\begin{align*}
	B_w e_n\coloneqq w_n e_{n-1}, \quad n\in\mathbb Z, 
\end{align*}
where $(e_n)_{n\in\mathbb Z}$ denotes the canonical base of $\ell^p(\mathbb Z)$, $1\leq p <+\infty$. It is obvious that $B_w$ defines a bounded linear operator with $\|B_w\|\leq \|w\|_{\infty}$.

We define the \emph{unilateral weighted backward shift} $B_w:\ell^p(\mathbb N)\to \ell^p(\mathbb N)$ in an analogous way with the obvious modifications.

It is known that a unilateral or a bilateral weighted backward shift on $\ell ^p (\mathbb Z)$, $1\leq p<+\infty$, is recurrent if and only if it is hypercyclic. See for example \cite{ChSe},\cite{Sece:thesis} and \cite{CP}. For a characterization of hypercyclic weighted shifts in terms of the weight sequence see for example \cite{Salas}.

As for rigidity for unilateral or bilateral weighted shifts there is not so much to talk about since these operators are never rigid. Indeed, if $B$ is a unilateral or bilateral weighted shift on $\ell^p$, say, then for every basis vector $e_k$ we have that $\|B^n e_k-e_k\|_p\geq 1$ if $n\geq 1$, thus $B$ cannot be rigid.
\subsection{Non separable Banach spaces.} As we mentioned above, there exist hypercyclic and thus recurrent weighted shifts on every $\ell^p$ with $1\leq p<\infty$. It turns out however that there are no recurrent weighted shifts on $\ell^\infty(\mathbb N)$ or $\ell^\infty(\mathbb Z)$. 

\begin{theorem}
	There does not exist recurrent unilateral or bilateral weighted backward shifts on $\ell^{\infty}(\mathbb{N})$ or $\ell^\infty(\mathbb Z)$, respectively. 
\end{theorem}

\begin{proof}
	Let $T:\ell^{\infty}(\mathbb N )\to \ell^\infty (\mathbb{N})$ be a unilateral backward weighted shift with weight sequence $a = (\alpha_n)_{n\in\mathbb N }$ and suppose that $T$ is recurrent. Let $M>1$ and define the vector $y\coloneqq(2,M,2,2,\ldots )$. Since $T$ is recurrent there exist a recurrent vector $x=(x_n)_{n\in\mathbb{N}}$ and a positive integer $l>1$ such that $\| x-y\|_{\infty}<\frac{1}{2}$ and
	\begin{align*}
		\| T^lx-x\|_{\infty}=\displaystyle{\sup_{j\geq 1} \left | \prod_{i=1}^{l}\alpha_{i+j-1}x_{l+j}-x_j\right | <\frac{1}{2}}.
	\end{align*}
	We get that
	\begin{align*}
		1<\displaystyle{\left | \prod_{i=1}^{l}\alpha_{i}x_{l+1}\right | <3}, \quad M-1<\displaystyle{\left | \prod_{i=1}^{l}\alpha_{i+1}x_{l+2}\right |}<M+1 
	\end{align*}
	and $\frac{3}{2}<|x_{l+1}|$, $|x_{l+2}|<\frac{5}{2}$. Using the previous estimates and since 
	\begin{align*}
		\frac{\displaystyle{\left | \prod_{i=1}^{l}\alpha_{i+1}x_{l+2}\right |}}{\displaystyle{\left | \prod_{i=1}^{l}\alpha_{i}x_{l+1}\right |}}=\frac{\alpha_{l+1}}{\alpha_1}\,\frac{|x_{l+2}|}{|x_{l+1}|}, 
	\end{align*}
	we arrive at 
	\begin{align*}
		\alpha_{l+1}>\frac{a_1}{5} (M-1). 
	\end{align*}
	Since $M$ can be chosen to be arbitrarily large, we conclude that the sequence $ (\alpha_n)_{n\in\mathbb{N}}$ is unbounded which is a contradiction. The proof for bilateral weighted shifts is essentially identical so we omit it. 
\end{proof}
A well known result of Salas from \cite{Salas} says that if $T$ is any unilateral weighted backward shift on $\ell^2(\mathbb N)$ then the operator $I+T$ is hypercyclic. In the non-separable case it is easy to see that $I+T$ can never be recurrent.

\begin{proposition}
	Let $T:\ell^\infty(\mathbb N)\to \ell^\infty(\mathbb N)$ be a unilateral weighted backward shift. Then $I+T$ is not recurrent. The same is true if $T:\ell^\infty(\mathbb Z)\to \ell^\infty(\mathbb Z)$ is a bilateral weighted backward shift. 
\end{proposition}

\begin{proof}
	Let $T:\ell^\infty(\mathbb N) \to \ell^\infty(\mathbb N)$ be a unilateral weighted backward shift with weight sequence $a=(a_n)_{n\in\mathbb N}$ and suppose that $I+T$ is recurrent. Let $y\coloneqq (1,1,\ldots,1,\ldots)$. There exist a vector $x\in\ell^\infty(\mathbb N)$ and a positive integer $N$ with $Na_1>5$ such that $\|x-y\|_\infty<\frac{1}{2}$ and 
	\begin{align*}
		\| (I+T)^N x-x\|_\infty= \sup_{j\geq 1} \ABs{ \bigg( \sum_{l=0} ^N \binom{N}{l}T^l x \bigg)_j -x_j} = \sup_{j\geq 1} \ABs{ \sum_{l=0} ^N \binom{N}{l}\big(\prod_{i=1}^{l}\alpha_{i+j-1}\big)x_{l+j}-x_j } <1. 
	\end{align*}
	Taking real parts in the previous inequality and $j=1$ in the supremum we must have 
	\begin{align*}
		\ABs{ \sum_{l=0} ^N \binom{N}{l}\big(\prod_{i=1}^{l}\alpha_i \big)\textnormal{Re} (x_{l+1}) -\textnormal{Re} (x_1)} <1. 
	\end{align*}
	The last implies that 
	\begin{align*}
		\ABs{ \sum_{l=0} ^N \binom{N}{l}\big(\prod_{i=1}^{l}\alpha_i \big)\textnormal{Re} (x_{l+1}) } <1+|\textnormal{Re} (x_1)|=1+\textnormal{Re} (x_1) . 
	\end{align*}
	Now observe that all the terms in the sum above are positive, so we conclude that 
	\begin{align*}
		N a_1\textnormal{Re}(x_2) < 1+\textnormal{Re} (x_1) . 
	\end{align*}
	Taking into account that $\textnormal{Re}(x_2)>1/2$ and $\textnormal{Re}(x_1)<3/2$, the last inequality above implies that $Na_1<5$, which contradicts the choice of $N$. A similar argument shows that $I+T$ is never recurrent when $T$ is a bilateral weighted backward shift. 
\end{proof}

\subsection{Diagonal operators.} We now turn to some examples of diagonal recurrent operators. We will see that diagonal operators on classical sequence spaces are recurrent if and only if they are rigid. There are however diagonal operators on $c_0(\mathbb N)$, for example, that are rigid without being uniformly rigid. See the discussion in Example \ref{ex.rignounif}.

\begin{theorem}
	\label{t.linfty} For a sequence $\lambda =(\lambda_1,\ldots,\lambda_k,\ldots )\in \ell^\infty(\mathbb N)$ we define the diagonal operator $T_\lambda:\ell^\infty(\mathbb N) \to \ell^\infty(\mathbb N)$ by the formula 
	\begin{align*}
		T_\lambda (x_1,x_2,\ldots,x_k,\ldots)\coloneqq (\lambda_1 x_1,\ldots,\lambda_k x_k,\ldots). 
	\end{align*}
	The following are equivalent: 
	\begin{itemize}
		\item [(i)] $T$ is recurrent. 
		\item [(ii)] $T$ is rigid. 
		\item [(iii)] $T$ is uniformly rigid. 
		\item [(iv)] For every $k\in\mathbb N$ we have $\lambda_k=e^{2\pi i \theta_k}$ for some $(\theta_k)_{k\in\mathbb N}\subset \R$ and 
		\begin{align*}
			\liminf_{n\to+\infty} \sup_{k\in\mathbb N}\abs{e^{2\pi i n \theta_k}-1} \to 0. 
		\end{align*}
	\end{itemize}
\end{theorem}

\begin{proof}
	First of all we claim that if $T_\lambda$ is recurrent then necessarily $|\lambda_k|=1$ for all $k\in\mathbb N$. Indeed, if $\overline{\Rec(T_\lambda)}=\ell^\infty(\mathbb N)$ then there exists a recurrent vector $y=(y_1,\ldots,y_k,\ldots)\in\ell^\infty(\mathbb N)$ with $y_k\neq 0$ for all $k\in\mathbb N$. Using this it is straightforward to show the claim. So, it suffices to consider unimodular $\lambda_k$'s, i.e., $\lambda_k=e^{2\pi i \theta_k}$ for some $(\theta_k)_{k\in\mathbb N}\subset \R$.
	
	Next, we observe that for every $n\in \mathbb N$ we have 
	\begin{align*}
		\norm T_\lambda ^n-I. . = \sup_{\|x\|_{\infty}\leq 1} \norm T_\lambda ^n x- x. \infty. \leq \sup_{k\in \mathbb N} \abs{e^{2\pi i n \theta_k}-1} \|x\|_\infty. 
	\end{align*}
	On the other hand, for $\textbf 1\coloneqq (1,1,\ldots,1,\ldots)\in \ell^\infty(\mathbb N)$ we have 
	\begin{align*}
		\norm T_\lambda ^n \textbf 1- \textbf 1. . = \sup_{k\in \mathbb N} \abs{e^{2\pi i n \theta_k}-1}. 
	\end{align*}
	We conclude that $ \norm T_\lambda ^n-I. . =\norm T_\lambda ^n \textbf 1- \textbf 1. . = \sup_{k\in \mathbb N} \abs{e^{2\pi i n \theta_k}-1}.$ Using this together with Proposition \ref{p.bimpliesunit}, (v), it is routine to show the equivalence of (i)-(iv). 
\end{proof}

\begin{theorem}
	For a sequence $\lambda \in \ell^\infty(\mathbb N)$ we define the operator $T_\lambda:X\to X$ as above, where $X=c_0(\mathbb N)$ or $X=\ell^p(\mathbb N)$, $1\leq p<+\infty$. 
	\begin{itemize}
		\item [(A)] The following are equivalent: 
		\begin{itemize}
			\item[(i)] $T_\lambda$ is recurrent 
			\item[(ii)] $T_\lambda$ is rigid. 
			\item[(iii)] For every $k\in\mathbb N$ we have $|\lambda_k|=1$. 
		\end{itemize}
		\item [(B)] The following are equivalent: 
		\begin{itemize}
			\item[(i)] $T_\lambda$ is uniformly rigid. 
			\item[(ii)]For every $k\in\mathbb N$ we have $\lambda_k=e^{2\pi i \theta_k}$ for some $(\theta_k)_{k\in\mathbb N}\subset \R$ and 
			\begin{align*}
				\liminf_{n\to+\infty} \sup_{k\in\mathbb N}\abs{e^{2\pi i n \theta_k}-1} \to 0. 
			\end{align*}
		\end{itemize}
	\end{itemize}
\end{theorem}

\begin{proof}
	We give the prove for the case $X=c_0(\mathbb N)$ since the proof for the case $X=\ell^p(\mathbb N)$ is essentially identical. The equivalence in (B) follows by exactly the same arguments as in the proof of Theorem~\ref{t.linfty}. We turn to the equivalences in (A). The only non-trivial thing to show is that (iii) implies (ii). To that end we fix a sequence $(\theta_k)_{k\in\mathbb N}\subset \R$ such that $\lambda_k=e^{2\pi i \theta_k}$ for all $k\in \mathbb N$. We exhibit the existence of a strictly increasing sequence of positive integers $(\rho_n)_{n\in N}$ satisfying $T_\lambda ^{\rho_n} \to I$ in the strong operator topology.
	
	For every positive integer $\ell$ we consider the set 
	\begin{align*}
		\{m^{(\ell)} _1 < m^{(\ell)} _2< \cdots < m^{(\ell)} _k <\cdots \}\coloneqq\big\{m\in\mathbb N:|e^{2\pi i m \theta_j}-1|<\frac{1}{2^\ell}\quad\text{for all}\quad j=1,\ldots,\ell\big\}. 
	\end{align*}
	Observe that for every $\ell\geq 2$ the sequence $(m_k ^{(\ell+1)})_{k\in\mathbb N}$ is a subsequence of $(m_k ^{(\ell)})_{k\in\mathbb N}$. Define $\rho_n\coloneqq m_n ^{(n)}$, for every $n\in \mathbb N$. The above construction easily implies that for every integer $M$ we have 
	\begin{align*}
		\lim_{n\to +\infty} \sup_{k\leq M} |e^{2\pi i \rho_n \theta_k }-1| =0. 
	\end{align*}
	Using this it is easy to see that $T_\lambda ^{\rho_n}x\to x$ as $n\to +\infty$, for every $x\in c_0 (\mathbb N)$. 
\end{proof}

\begin{example}\label{ex.rignounif}
	For $\theta \in \R$ consider the sequence $(\lambda_k)_{k\in\mathbb N}$ with $\lambda_k\coloneqq e^{2\pi i k\theta}$ for all $k\in\mathbb N$. If $\theta \in \mathbb Q$ it is easy to check that the sequence $(e^{2\pi i k\theta})_{k\in\mathbb N}$ satisfies the condition in Theorem~\ref{t.linfty}, (iv), and thus the corresponding diagonal operator $T_\lambda$ is uniformly rigid on $c_0(\mathbb N)$ and $\ell^p(\mathbb N)$, $1\leq p\leq +\infty$. On the other hand if $\theta \in \R\setminus \mathbb Q$ then $ \sup_{k\in\mathbb N}\abs{e^{2\pi i n k\theta }-1}=2$ for every $n\in\mathbb N$ and, therefore, $T_\lambda$ fails to be uniformly rigid on any of the spaces considered above. However, $T_\lambda$ is still rigid on $c_0(\mathbb N)$ or $\ell^p(\mathbb N)$, $1<p<+\infty$. 
\end{example}

\section{Composition operators} In this section we turn to the study of composition operators in different spaces of functions and characterize when these operators are recurrent in terms of conditions on their symbol. If $Y$ is a Banach or Fr\'echet space of functions $f:A\to \C$, where $A\subset \C$, and $\phi:A\to A$, we will denote the composition operator with \emph{symbol} $\phi$ by $C_\phi:Y\to Y$: 
\begin{align*}
	C_\phi(f)=f\circ \phi,\quad f\in Y, 
\end{align*}
whenever this operator is well defined.

\subsection{Composition operators on the space of continuous functions \texorpdfstring{$C([0,1])$}{C([0,1])}.}\label{s.compcont}

Let $C([0,1])$ denote the space of continuous functions $f:[0,1]\to \mathbb C$ equipped with the topology of uniform convergence. For any continuous function $\phi:[0,1]\to [0,1]$ the composition operator 
\begin{align*}
	C_\phi:C([0,1])\to C([0,1]), \quad C_\phi(f)(x)\coloneqq f(\phi(x)),\quad x\in[0,1], 
\end{align*}
is a well defined bounded linear operator. The first easy observation is that a necessary condition for $C_\phi$ to be recurrent is that $\phi$ is \emph{one-to-one}. Indeed, supposing that it is not, there exist $x_1,x_2\in[0,1]$ with $x_1\neq x_2$ such that $\phi(x_1)=\phi(x_2)$. We get that for every recurrent vector $f\in \Rec(C_\phi)$ we must have $f(x_1)=f(x_2)$ and thus that 
\begin{align*}
	C([0,1])=\overline{\Rec(C_\phi)}\subset \overline{\{f\in C([0,1]):f(x_1)=f(x_2)\}} , 
\end{align*}
which is clearly a contradiction. Furthermore it is easy to see that $\phi:[0,1]\to[0,1]$ must be \emph{onto} $[0,1]$. However, this is an easy exercise: since $\phi$ is one-to-one, if it is not onto $[0,1]$ then there must be some interval $(a,b)\subset [0,1]$ that avoids the range of $\phi$. Constructing a non-zero continuous function with compact support inside $(a,b)$ immediately leads to a contradiction. Thus $\phi$ is a strictly increasing from $[0,1]$ onto $[0,1]$ with $\phi(0)=0$ and $\phi(1)=1$ or a strictly decreasing function from $[0,1]$ onto $[0,1]$ with $\phi(0)=1$ and $\phi(1)=0$.

If $\phi$ is strictly increasing we claim that the only possibility is the identity $\phi(x)=x$. If not then there is some $x_o\in(0,1)$ and some $\delta>0$ such that $\phi(x_o)<x_o-\delta$ or $\phi(x_o)>x_o+\delta$. Without loss of generality let us assume that $\phi(x_o)<x_o-\delta$ so that $\phi^{[n]}(x_o)<x_o-\delta$ for all positive integers $n\in\mathbb N$. Now we observe that $\phi^{[n]}(x_o)<\phi^{[n-1]}(x_o)$ for all $n\geq 1$. To see this for $n=1$ we remember that $\phi(x_o)<x_o-\delta$ and that $\phi$ is strictly increasing so that $\phi(\phi(x_o)) <\phi(x_o-\delta)<\phi(x_o)$. The proof for general $n$ follows easily by induction, using the fact that $\phi$ is strictly increasing. Thus $(\phi^{[n]})_{n\in\mathbb N}$ is strictly decreasing and bounded from below and we can conclude that there exists $y< x_o$ such that $\phi^{[n]}(x_o)\to y$ as $n\to +\infty$. We get that $f(\phi^{[n]}(x_o))\to f(y)$ for every $f\in C([0,1])$. Now let $f\in \Rec(C_\phi)$ with corresponding sequence $(k_n)_{n\in\mathbb N}$. We have that 
\begin{align*}
	f(\phi^{[k_n]}(x_o))\to f(x_o)\text{ as }n\to+\infty. 
\end{align*}
Thus $f(x_o)=f(y)$ for all recurrent vectors $f\in\Rec(C_\phi)$. Since $x_o\neq y$ and $\overline{\Rec(C_\phi)}=C([0,1])$ this is clearly a contradiction.

If $\phi$ is strictly decreasing we use the analogous argument to show that the only choice is $\phi(x)=1-x$. We have thus showed the following. 

\begin{theorem}
	\label{t.compcont} Suppose that $C_\phi:C([0,1])\to C([0,1])$ is a composition operator. The following are equivalent. 
	\begin{itemize}
		\item[(i)] $C_{\phi}$ is recurrent. 
		\item[(ii)] $C_{\phi }$ is rigid. 
		\item[(iii)] $C_{\phi}$ is uniformly rigid. 
		\item[(iv)]$\phi(x)=x$ or $\phi(x)=1-x$. 
	\end{itemize}
\end{theorem}

\begin{remark}
	Observe that the spectrum of $C_\phi:C([0,1])\to C([0,1])$ for $\phi(x)=1-x$ is $\sigma (C_\phi )=\{ -1, 1\}$. Indeed, since ${C_\phi}^2=I$ we have that $\sigma (C_\phi ) \subset \{ -1, 1\}$. Clearly $1\in \sigma_p (C_\phi )$. For $f(x)=e^{\pi ix}+e^{-\pi ix}$ we have $f(1-x)=-f(x)$, hence $-1\in \sigma (C_\phi )$. In fact we have $\sigma_p(C_\phi )=\sigma (C_\phi )=\{ -1,1\}$. Therefore, the only possible structures for the spectrum of a composition operator on $C([0,1])$ which is recurrent are of the form $\{ 1\}$, $\{ -1,1\}$. The above theorem remains valid, without changing the proof, by replacing the space $C([0,1])$ with the space $C_{\mathbb R}([0,1])$ of real valued continuous functions on $[0,1]$. In this case, the spectrum of the composition operator $C_\phi:C_{\mathbb R}([0,1])\to C_{\mathbb R}([0,1])$ for $\phi(x)=1-x$, which coincides with the point spectrum, is also the set $\{ -1, 1\}$. To see this just consider the real valued function $f(x)=\frac{1}{2}-x$, $x\in [0,1]$ and observe that $f(1-x)=-f(x)$, $x\in [0,1]$. 
\end{remark}

\begin{remark}
	As we have already mentioned T. Eisner showed in \cite{EIS} that if $T:X\to X$ is a power bounded operator acting on a complex \emph{separable} Banach space $X$ and 
	\begin{align*}
		\overline{\textnormal{span} \{x\in X: Tx=\lambda x\quad \text{for some} \quad \lambda \in {\mathbb T} \}}=X, 
	\end{align*}
	then $T$ is rigid. The composition operator $C_\phi :C([0,1])\to C([0,1])$ for $\phi (x)=x$ satisfies trivially the assumptions in Eisner's theorem. Let us check that this is also the case for the composition operator $C_\phi :C([0,1]\to C([0,1]$ induced by $\phi (x)=1-x$ (which is trivially rigid since $C_\phi ^2=I$). For every positive integer $n$ the function $f_n(x)=(\frac{1}{2}-x)^n$, $x\in [0,1]$ satisfies $f_n(1-x)=f_n(x)$, $x\in [0,1]$ if $n$ is even and $f_n(1-x)=-f_n(x)$, $x\in [0,1]$ if $n$ is odd. Since every non-zero constant function on $[0,1]$ is an eigenfunction for $C_\phi$ corresponding to the eigenvalue $1$ and $f_1$ is also an eigenfunction for $C_\phi$ corresponding to the eigenvalue $-1$ we conclude that the monomial $p_1(x)=x$ belongs to $\textnormal{span} \{\Ker(C_\phi -I)\cup \Ker(C_\phi+I)\}$. The polynomial $f_2$, which has degree two, belongs to the vector space $\textnormal{span}\{\Ker(C_\phi -I)\cup \Ker(C_\phi+I)\}$ and since the non-zero constant functions and the monomial $p_1(x)=x$ also belong to $\textnormal{span}\{\Ker(C_\phi -I)\cup \Ker(C_\phi+I)\}$ it follows that $p_2(x)=x^2 \in \textnormal{span} \{\Ker(C_\phi -I)\cup \Ker(C_\phi+I)\}$. Continuing in the same way we get 
	\begin{align*}
		\{ p_n(x)=x^n: n=0,1,2,\ldots \} \subset \textnormal{span}\{\Ker(C_\phi -I)\cup \Ker(C_\phi+I)\} , 
	\end{align*}
	which in turn implies that every polynomial belongs to $\textnormal{span} \{ \Ker(C_\phi -I)\cup \Ker(C_\phi+I)\}$. The conclusion now follows by the Weierstrass approximation theorem.
	
	Observe that, from Eisner's result and the previous discussion, each of the four equivalent statements in Theorem~\ref{t.compcont} is equivalent to: 
	\begin{align*}
		\overline{ \textnormal{span}\{ f\in C([0,1]): C_\phi f=\lambda f \quad \text{for some} \quad \lambda \in {\mathbb T} \}}=C([0,1]). 
	\end{align*}
\end{remark}

\subsection{Composition operators on the space of entire functions.} Let $H(\mathbb C)$ denote the Fr\'echet space of entire functions endowed with the topology of uniform convergence on compact sets, and let $\phi\in H(\mathbb C)$. We will abbreviate ``holomorphic and one-to-one'' by ``univalent''. The \emph{composition operator} induced by $\phi$ is defined on $H(\mathbb C)$ as $C_\phi(f)=f\circ \phi$. Obviously $C_\phi$ is continuous and linear. It is well known, and easy to prove, that $C_\phi$ is hypercyclic if and only if $\phi$ is of the form $\phi(z)=z+b$ with $b\neq 0$. See \cite{LuMo}. The class of recurrent composition operators on $H(\mathbb C)$ turns out to be slightly wider.

\begin{theorem}
	\label{t.h(c)} Consider the composition operator $C_\phi:H(\mathbb C)\to H(\mathbb C)$ for some $\phi \in H(\mathbb C)$. The following are equivalent. 
	\begin{itemize}
		\item[(i)] $C_{\phi}$ is recurrent. 
		\item[(ii)] $\phi(z)=az+b$ with $a,b\in\mathbb C$ and $|a|=1$. 
	\end{itemize}
\end{theorem}

\begin{proof}
	Assume first that $C_\phi$ is recurrent. By an argument similar to the one used \S~\ref{s.compcont} for $C([0,1])$ it follows that $\phi$ must be univalent. However, the only univalent entire functions are of the form $\phi(z)=az+b$ for some $a,b\in\mathbb C$.
	
	Now for $\phi(z)=az+b$ and $n\in \mathbb N$ we have that 
	\begin{align*}
		\phi^{[n]}(z)= 
		\begin{cases}
			a^nz+\frac{a^n-1}{a-1}b, \quad &a\neq 1,\\
			z+nb,\quad &a=1 . 
		\end{cases}
	\end{align*}
	If $a=1$ and $b\neq 0$ then $C_\phi$ is the translation by $b$ which is known to be hypercyclic and thus recurrent. On the other hand if $a=1$ and $b=0$ then $C_\phi$ is the identity which is obviously recurrent. If $a\neq 1$ with $|a|=1$ then we will show that every $f\in H(\mathbb C)$ is a recurrent vector for $C_\phi$ and thus $C_\phi$ is recurrent. Indeed, take any $f\in H(\mathbb C)$. We need to check that for every $\epsilon,R>0$ and every positive integer $N>0$ there exists $n>N$ such that, 
	\begin{equation}
		\label{e.lessepsilon} \sup_{z\in \overline{D(0,R)}} \ABs{f\big(a^nz+\frac{a^n-1}{a-1}b\big)-f(z)}<\epsilon. 
	\end{equation}
	By the uniform continuity of $f$ on compact subsets of $\mathbb C$ there exists $\delta>0$ such that: 
	\begin{align*}
		\text{if}\quad z,w\in \overline{D\left(0,R+\frac{2|b|}{|a-1|}\right)}\quad \text{and}\quad |z-w|<\delta,\ \text{then}\quad\Abs{f(z)-f(w)}<\epsilon. 
	\end{align*}
	There exists $n>N$ such that $|a^n-1|<{\delta}{(R+\frac{|b|}{|a-1|})^{-1}}.$ Then, for every $z\in \overline{D(0,R)}$ we have that $z,a^nz+\frac{a^n-1}{a-1}b\in \overline{D(0,R+\frac{2|b|}{|a-1|})}$ and 
	\begin{align*}
		\ABs{a^nz+\frac{a^n-1}{a-1}b -z }\leq |a^n-1|\big( R+ \frac{|b|}{|a-1|}\big)<\delta, 
	\end{align*}
	and \eqref{e.lessepsilon} follows.
	
	Finally, for $|a|<1$ we have that $a^nz+\frac{a^n-1}{a-1}b \to -\frac{b}{a-1}$ uniformly on compact subsets of $\mathbb C$ and a simple argument shows that $C_\phi$ cannot be recurrent. Similarly, if $|a|>1 $ then take $R>0$ sufficiently large so that $0\in D(\frac{b}{a-1},R)$. Then for any strictly increasing sequence of positive integers $(k_n)_{n\in\mathbb N}$ and every non-constant entire function $f$ we have 
	\begin{align*}
		\sup_{z\in\overline{D(0,R)}}\ABs {f(a^{k_n}z+\frac{a^{k_n}-1}{a-1}b)-f(z)}=\sup_{z\in \overline{D(\frac{b}{a-1},R)}}\Abs{f(a^{k_n}z-\frac{b}{a-1})-f(z-\frac{b}{a-1})}\to+\infty 
	\end{align*}
	as $n\to+\infty$. Thus $C_\phi$ cannot be recurrent in this case either. 
\end{proof}

\begin{theorem}
	\label{t.h(rigid)} Consider the composition operator $C_\phi:H(\mathbb C)\to H(\mathbb C)$ for some $\phi \in H(\mathbb C)$. The following are equivalent. 
	\begin{itemize}
		\item[(i)] $C_{\phi}$ is rigid. 
		\item[(ii)] $\phi(z)=az+b$ where, either $a=1$ and $b=0$ or $a\in \mathbb T\setminus \{ 1\}$ and $b\in \mathbb C$ . 
	\end{itemize}
\end{theorem}

\begin{proof}
	Assume that $C_\phi$ is rigid. Then $C_\phi$ is recurrent and by Theorem~\ref{t.h(c)}, $\phi$ should be of the form $\phi(z)=az+b$ with $a,b\in\mathbb C$ and $|a|=1$. Let us first exclude the case $a=1$ and $b\neq 0$. Indeed, if $\phi (z)=z+b$ with $b\neq 0$ then it is well known that $C_{\phi}$ is hereditarily hypercyclic thus $C_{\phi}$ is not rigid. On the other hand if $a=1$ and $b=0$, i.e. $\phi (z)=z$, then $C_{\phi }$ is rigid, trivially. It only remains to handle the case $a\in \mathbb T\setminus \{ 1\}$ and $b\in \mathbb C$. Fix such $a$ and $b$ and then fix a strictly increasing sequence of positive integers $(k_n)_{n\in\mathbb N}$ such that 
	\begin{align*}
		|a^{k_n}-1|<\frac{1}{n} \quad \text{for every} \quad n=1,2,\ldots . 
	\end{align*}
	Take any $f\in H(\mathbb C)$. We shall prove that for every $\epsilon,R>0$ there exists a positive integer $N$ such that for every $n\geq N$ 
	\begin{equation*}
		\sup_{z\in \overline{D(0,R)}} \ABs{f\big(a^{k_n}z+\frac{a^{k_n}-1}{a-1}b\big)-f(z)}<\epsilon. 
	\end{equation*}
	By the uniform continuity of $f$ on compact subsets of $\mathbb C$ there exists $\delta>0$ such that:
	\begin{align*}
		\text{if } z,w\in \overline{D\left(0,R+\frac{2|b|}{|a-1|}\right)}\text{ and } |z-w|<\delta \text{ then }\Abs{f(z)-f(w)}<\epsilon. 
	\end{align*}
	Fix a positive integer $N$ such that 
	\begin{align*}
		\frac{1}{N}<{\delta}{(R+\frac{|b|}{|a-1|})^{-1}}. 
	\end{align*}
	Then 
	\begin{align*}
		|a^{k_n}-1|<{\delta}{(R+\frac{|b|}{|a-1|})^{-1}} \quad \text{for every } \quad n\geq N 
	\end{align*}
	and arguing as in the proof of the Theorem~\ref{t.h(c)} the conclusion follows. 
\end{proof}
\subsection{Composition operators on the space of holomorphic functions in the punctured plane.} Let ${\mathbb C}^*:= \mathbb C \setminus \{ 0 \}$ denote the punctured plane and $H({\mathbb C}^*)$ denote the Fr\'echet space of holomorphic function on $\C^*$ endowed with the topology of uniform convergence on compact sets of $\C^*$. Then the automorphisms of ${\mathbb C}^*$ are the functions 
\begin{align*}
	\phi (z)=az \quad \text{or} \quad \phi (z)=\frac{a}{z}, \quad a\in \mathbb C\setminus \{ 0\} . 
\end{align*}

\begin{theorem}
	\label{t.punct-rec} Let $\phi$ be an automorphism of ${\mathbb C}^*$. Then the composition operator $C_{\phi }: H({\mathbb C}^*) \to H({\mathbb C}^*)$ is recurrent if and only if either $\phi (z)=az$ with $a\in \mathbb{T}$ or $\phi (z)=\frac{a}{z}$ with $a\in {\mathbb C}^*$. 
\end{theorem}

\begin{proof}
	Let $\phi (z)=az$ with $|a|\neq 1$. Suppose that $C_{\phi }$ is recurrent. Since $\Rec(C_{\phi })$ is dense in $H({\mathbb C}^*)$, there exists $f(z)=\sum_{n\in \mathbb Z}c_nz^n \in \Rec(C_{\phi })$ with $c_{-1}\neq 0$. We have 
	\begin{align*}
		\int_{\mathbb T}((C_{\phi })^nf(z)-f(z))dz=\int_{\mathbb T}(f(a^nz)-f(z))dz=2\pi i(\frac{c_{-1}}{a^n}-c_{-1}), 
	\end{align*}
	where the unit circle $\mathbb T$ is positively oriented. Since $f\in \Rec(C_{\phi })$ there exists a strictly increasing sequence of positive integers $(k_n)_{n\in\mathbb N}$ such that the left hand side of the above equality tends to zero and since $c_{-1}\neq 0$ we conclude that $a^{k_n}\to 1$, which is a contradiction. Consider now $\phi (z)=az$ with $|a|=1$. Then there exists a strictly increasing sequence of positive integers $(k_n)_{\mathbb N}$ such that $a^{k_n}\to 1$ and now it is easy to show that for every $f\in H({\mathbb C}^* )$, $(C_{\phi })^{k_n}f\to f$ uniformly on compact subsets of ${\mathbb C}^*$, i.e., $C_{\phi }$ is recurrent. It remains to handle the case $\phi (z)=\frac{a}{z}$, $a\in {\mathbb C}^*$. For this observe that $\phi \circ \phi (z)=z$, $z\in {\mathbb C}^*$ and therefore $(C_{\phi })^{2n}f=f$ for every $n\in \mathbb{N}$ and every $f\in H({\mathbb C}^* )$. Therefore $C_{\phi }$ is recurrent and this completes the proof of the theorem. 
\end{proof}

\begin{theorem}
	Let $\phi$ be an automorphism of ${\mathbb C}^*$. Then the composition operator $C_{\phi }: H({\mathbb C}^*) \to H({\mathbb C}^*)$ is rigid if and only if either $\phi (z)=az$ with $a\in \mathbb{T}$ or $\phi (z)=\frac{a}{z}$ with $a\in {\mathbb C}^*$. 
\end{theorem}

\begin{proof}
	A careful inspection of the proof of Theorem \ref{t.punct-rec} gives the desired result. 
\end{proof}

\subsection{Composition operators on the space of holomorphic functions on the unit disk.} We now consider the Fr\'echet space $H(\mathbb D)$ of holomorphic functions on the unit disk $\mathbb D$, endowed with the topology of uniform convergence on the compact subsets of $\mathbb D$. Let $\phi:\mathbb D\to \mathbb D$ be a holomorphic function. We define the composition operator $C_\phi:H(\mathbb D)\to H(\mathbb D)$ as $C_\phi(f)=f\circ \phi$. Again it is clear that $C_\phi$ is a continuous linear operator.

We will be especially interested on composition operators induced by \emph{linear fractional maps}.

\begin{definition}
	A non-constant map $\phi:\mathbb D\to \mathbb D$ is called a \emph {linear fractional map} if it can be written in the form 
	\begin{align*}
		\phi(z)=\frac{az+b}{cz+d},\quad z\in\mathbb D, 
	\end{align*}
	for some $a,b,c,d\in\mathbb C$ satisfying $ad-bc\neq 0$. The latter condition is necessary and sufficient for $\phi$ to be nonconstant. We denote by $\lfm$ the set of all linear fractional maps of $\mathbb D$ into itself. The linear fractional maps that take $\mathbb D$ \emph{onto} itself are called \emph{conformal automorphisms of $\mathbb D$}. 
\end{definition}

The members of $\lfm$ have at least one and at most two fixed points in $\hat{ \mathbb C}$. We classify them as: 
\begin{itemize}
	\item{Linear fractional maps without a fixed point in $\mathbb D$.} 
	\begin{itemize}
		\item [-] \emph{parabolic linear fractional maps:} those having a unique attractive fixed point on $ \mathbb T$. 
		\item [-] \emph{hyperbolic maps with attractive fixed point on $\mathbb T$:} those having an attractive fixed point $\alpha\in\mathbb T$ and a second fixed point $\beta\in\hat{\mathbb C}\setminus \mathbb D$. The linear fractional map is a hyperbolic automorphism of $\mathbb D$ if and only if both fixed points are on $\mathbb T$.
	\end{itemize}
	\item{Linear fractional maps having a fixed point in $\mathbb D$.} Here there are two cases:
	\begin{itemize}
		\item [-] either \emph{the interior fixed point is attractive}, or
		\item [-] the map is an \emph{elliptic automorphism:} The automorphisms of $\mathbb D$ having a fixed point $\alpha\in \mathbb D$ and the second fixed point $\beta\in \hat{\mathbb C}\setminus \overline{\mathbb D}$. 
	\end{itemize}
\end{itemize}

For these notions and classification we refer the reader to \cite{ShapB}.

The following theorem describes the recurrent composition operators, induced by holomorphic self-maps of the disk.

\begin{theorem}
	\label{t.compHD} Let $\phi:\mathbb D\to\mathbb D$ be a holomorphic function. The composition operator $C_\phi:H(\mathbb D)\to H(\mathbb D)$ is recurrent if and only if either $\phi$ is univalent and has no fixed point in $\mathbb D$ or $\phi$ is an elliptic automorphism.  
\end{theorem}

\begin{proof}
	First we assume that $C_\phi$ is recurrent. Using the the same argument as in the proof of Theorem~\ref{t.h(c)}, (i), we see that $\phi$ is necessarily univalent. If $\phi$ has no fixed point in $\mathbb D$ then there is nothing to show. Assume now that $\phi$ has an interior fixed point $p\in\mathbb D$. If $\phi$ is an automorphism of the disk then it is necessarily an elliptic automorphism; see \cite{ShapB}. If $\phi$ is not an elliptic automorphism then the Denjoy-Wolff Iteration Theorem, \cite{ShapB}*{Proposition 1, Chapter 5}, implies that $\phi^{[n]}$ converges to $p$ uniformly on compact subsets of $\mathbb D$. We conclude that the only limit points of the $C_\phi$-orbit are the constant functions. Therefore $C_\phi$ cannot be recurrent in this case.
	
	To show the converse observe that if $\phi$ is an elliptic automorphism then $\phi$ is conjugate to a rotation; see \cite{ShapB}*{Chapter 0}. Thus, there exists a linear fractional map $S$ and a complex number $\lambda\in \mathbb T$ such that $C_\phi=S^{-1} C_{\phi_\lambda} S$ where $\phi_\lambda (z)=\lambda z$, $z\in\mathbb D$. As in the proof of Theorem~\ref{t.h(c)} it is easy to see that $C_{\phi_\lambda}$ is recurrent and by Lemma \ref{l.equiv} we get that $C_\phi$ is recurrent. If $\phi$ is univalent and has no fixed point in $\mathbb D$ then by the Denjoy-Wolff theorem, \cite{ShapB}*{p. 78, Chapter 5}, there is a point $w\in \mathbb T$ such that $\phi^{[n]}\to w$ uniformly on compact subsets of $\mathbb D$. This implies that for every compact set $K\subset \mathbb D$ there exists a positive integer $n$ such that $ \phi^{[n]}(K)\cap K=\emptyset$. By \cite{ERMO}*{Theorem 3.2} $C_\phi$ is hypercyclic and thus recurrent.
\end{proof}
Specializing to linear fractional maps immediately gives the following corollary.

\begin{corollary} Let $\phi\in\lfm$. The composition operator $C_\phi:H(\mathbb D)\to H(\mathbb D)$ is recurrent if and only if $\phi$ is either parabolic, or hyperbolic with no fixed point in $\mathbb D$, or an elliptic automorphism.
\end{corollary}

We now characterize the rigid composition operators on $H(\mathbb D)$.

\begin{theorem}
	Let $\phi:\mathbb D\to \mathbb D$ be holomorphic. The composition operator $C_\phi:H(\mathbb D)\to H(\mathbb D)$ is rigid if and only if $\phi$ is an elliptic automorphism. 
\end{theorem}

\begin{proof}
	If $\phi$ is an elliptic automorphism we note as in the proof of Theorem \ref{t.compHD} above that there exists a linear fractional map $S$ and a complex number $\lambda\in \mathbb T$ such that $C_\phi=S^{-1} C_{\phi_\lambda} S$ where $\phi_\lambda (z)=\lambda z$, $z\in\mathbb D$. It is now an easy exercise to check that $C_{\phi_\lambda}$ is rigid. Assume now that $C_\phi$ is rigid. Then Theorem \ref{t.compHD} implies that either $\phi$ has no fixed point in $\mathbb D$ or that it is an elliptic automorphism. However, if $\phi$ has no fixed point in $\mathbb D$ it follows as in the proof of Theorem \ref{t.compHD} that for every compact set $K\subset \mathbb D$ there exists a positive integer $n_o$ such that $ \phi^{[n]}(K)\cap K=\emptyset$ for all $n\geq n_o$. This implies that $C_\phi$ is hereditarily hypercyclic; see for instance \cite{ERMO}*{Theorem 3.2}. In this case $C_\phi$ cannot be rigid. Thus $\phi$ is an elliptic automorphism. 
\end{proof}

\subsection{Composition operators on the Hardy space \texorpdfstring{$H^{2}(\mathbb D)$}{H2(D)}.}

In what follows we consider composition operators on the Hardy space $H^2(\mathbb D)$, consisting of holomorphic functions $f:\mathbb D\to \C$ such that
\begin{align*}
	\| f\|_{H^2(\mathbb D)}\coloneqq \sup_{0\leq r<1}{\left( \frac{1}{2\pi }\int_0^{2\pi }|f(re^{i\theta })|^2 d\theta \right) }^{1/2}<+\infty . 
\end{align*}
We immediately restrict our attention to the special class of symbols $\phi\in \lfm$.

\begin{theorem}
	\label{t.compH2} Let $\phi\in \lfm$. Then the operator $C_\phi :H^2(\mathbb D)\to H^2(\mathbb D)$ is recurrent if and only if $\phi$ is either hyperbolic with no fixed point in $\mathbb D$, or a parabolic automorphism, or an elliptic automorphism. 
\end{theorem}

\begin{proof}
	We first consider the case that $\phi\in\lfm$ has no fixed points in $\mathbb D$. Then $\phi$ is either parabolic or hyperbolic. In either case, \cite{BM}*{Theorem 1.47} implies that $C_\phi$ is hypercyclic and thus recurrent. If $\phi$ is a parabolic non-automorphism then by \cite{ShapB}*{The Linear Fractional Hypercyclicity Theorem, p.114}, only constant functions can be limit points of $C_\phi$-orbits. Therefore $C_\phi$ is not recurrent.
	
	Assume now that $\phi$ has a fixed point $p\in\mathbb D$. If $\phi$ is not an elliptic automorphism of $\mathbb D$ we claim that $C_\phi$ is not recurrent. Indeed, assume, for the sake of contradiction, that $C_\phi$ is recurrent and consider a non-constant $f\in\Rec(C_\phi)$. Then there exists a strictly increasing sequence of positive integers $(k_n)_{n\in\mathbb N}$ such that 
	\begin{align*}
		\Norm f\circ \phi^{[k_n]}-f.H^2(\mathbb D). \to 0,\quad n\to +\infty. 
	\end{align*}
	Therefore $f\circ \phi^{[k_n]}$ converges to $f$ uniformly on compact subsets of $\mathbb D$. Let $0<r<1$. Then 
	\begin{align*}
		\sup_{|z|\leq r}\Abs{f(\phi^{[k_n]}(z))-f(z)}\to 0,\quad n\to+\infty. 
	\end{align*}
	Since we have assumed that $\phi$ is not an elliptic automorphism, the interior fixed point $p$ must be attractive. By \cite{ShapB}*{Proposition 1, Chapter 5} the iterates $\phi^{[n]}$ converge to $p$ uniformly on compact subsets of $\mathbb D$. We conclude that for every $w\in \overline{D(0,r)}$, $f(\phi^{[k_n]})(w))\to f(p) $ as $n\to+\infty$. It readily follows that $f(w)=f(p)$ for every $w\in \overline{D(0,r)}$ and thus $f$ is constant, a contradiction. It only remains to check what happens when $\phi$ is an automorphism of the disk in which case $\phi$ is an elliptic automorphism. Without loss of generality we can assume that $\phi$ is of the form $\phi(z)=\lambda z$ for some $\lambda\in\mathbb T$ and a direct computation shows that every $f\in H^2(\mathbb D)$ is a recurrent vector for $C_\phi$. 
\end{proof}

\begin{theorem}
	Let $\phi\in \lfm$. Consider the composition operator $C_\phi:H^2(\mathbb D)\to H^2(\mathbb D)$. 
	\begin{itemize}
		\item [(i)] $C_\phi$ is rigid if and only if $\phi$ is an elliptic automorphism. 
		\item [(ii)] $C_\phi$ is uniformly rigid if and only if $\phi$ is conjugate to a rational rotation. 
	\end{itemize}
\end{theorem}

\begin{proof}
	For (i) first assume that $C_\phi$ is rigid. By Theorem \ref{t.compH2} the operator $C_\phi$ is recurrent and thus $\phi$ is either hyperbolic with no fixed point in $\mathbb D$, or a parabolic automorphism, or an elliptic automorphism. If $\phi$ is hyperbolic with no fixed point in $\mathbb D$ or a parabolic automorphism then $C_\phi$ is hereditarily hypercyclic; see \cite{BM}. Thus $C_\phi$ cannot be recurrent. To complete the proof of (i) it remains to show that if $\phi$ is an elliptic automorphism then $C_\phi$ is rigid. Without loss of generality we can assume that $\phi(z)=\lambda z$ for some $\lambda \in \mathbb T$. There exists a strictly increasing sequence of positive integers $(k_n)_{n\in\mathbb N}$ such that $\lambda^{k_n}\to 1$. For any $f(z)=\sum_{m\geq 0} a_m z^m \in H^2(\mathbb D)$ we have 
	\begin{align}
		\| C_\phi ^{k_n}f-f\|_{H^2(\mathbb D)} ^2 = \sum_{m\geq 0} |a_m|^2|1-\lambda^{m k_n}|^2\to 0\quad\text{as}\quad n\to+\infty 
	\end{align}
	by dominated convergence. Thus $C_\phi$ is rigid.
	
	For the proof (ii) we observe that if either of the equivalences is true then $\phi$ is necessarily conjugate to a rotation so we restrict our attention to $\phi$ of the form $\phi(z)=\lambda z$, with $|\lambda|=1$. For any positive integer $n$ we have 
	\begin{align*}
		\|C_\phi ^{n}- I\| = \sup_{m\geq 0} |1-\lambda^{m n}|. 
	\end{align*}
	However, $\liminf_{n\to+\infty}\sup_{m\geq 0}|1-\lambda^{m n}|= 0$ if and only if $\lambda=e^{2\pi i \theta}$ with $\theta \in \mathbb Q$. 
\end{proof}

\section{Multiplication Operators}\label{s.mult} We now consider multiplication operators on different spaces of functions. For some Banach or Fr\'echet space $Y$ we will denote by $M_\phi:Y\to Y$ the multiplication operator with \emph{symbol} $\phi$, that is, 
\begin{align*}
	M_\phi(f)=\phi f, \quad f\in Y, 
\end{align*}
whenever this operator is well defined. In most cases we will see that the recurrent multiplication operators are, in some sense, trivial, meaning that the symbol of the operator is a constant function. 

\subsection{Multiplication operators on spaces of continuous functions.} First we consider multiplication operators on spaces of continuous functions.

\begin{theorem}
	Let $(K,d)$ be a compact and connected metric space and denote by $C(K)$ the continuous functions $f:K\to \C$. For $\phi \in C(K)$ the following are equivalent. 
	\begin{itemize}
		\item[(i)] $M_\phi:C(K)\to C(K)$ is recurrent. 
		\item[(ii)] $M_\phi:C(K)\to C(K)$ is rigid. 
		\item[(iii)] $M_\phi:C(K)\to C(K)$ is uniformly rigid. 
		\item [(iv)] There exists $a\in \mathbb T$ such that $\phi (x)=a$ for every $x\in K$. 
	\end{itemize}
\end{theorem}

\begin{proof}
	The implications $\text{(iv)}\Rightarrow \text{(iii)}$, $\text{(iii)}\Rightarrow \text{(ii)}$, $\text{(ii)}\Rightarrow \text{(i)}$ are trivial to show. In order to show that (i) implies (iv) let us assume that $M_\phi$ is recurrent. Suppose that there exist $x_0\in K$ with $|\phi (x_0)|\neq 1$ and consider any $f\in \Rec(M_\phi )$. If $|\phi (x_0)|<1$ then, by the continuity of $\phi$, there exists an open ball $B(x_0,\delta )=\{ y\in K: d(y,x_0)<\delta \}$ such that $|\phi (y)|<1$ for every $y\in B(x_0,\delta )$. From the last and the fact that $f\in \Rec(M_\phi )$ it easily follows that $f(y)=0$ for every $y\in B(x_0,\delta )$. Therefore we have 
	\begin{align*}
		\Rec(M_\phi )\subset \{ f\in C(K): f(y)=0 \quad\text{for every}\quad y\in B(x_0,\delta ) \} . 
	\end{align*}
	However the right hand set in the above inclusion cannot be dense in $C(K)$ (there is hidden here an argument involving the connectedness of the space $K$ which is left to the reader), thereby giving a contradiction. The case $|\phi (x_0)|>1$ can be handled in a similar fashion and we leave the details to the reader. At this point we know that $|\phi (x)|=1$ for every $x\in K$, which in turn implies that $\| M_\phi \| =1$. Therefore $M_\phi $ is power bounded and we get that $\Rec(M_\phi )=C(K)$. Assume that $\phi (x_1)\neq \phi (x_2)$ for some $x_1,x_2\in K$. The set $K$ is connected, hence $\phi (K)$ is connected and since $|\phi (x_1)|=|\phi (x_2) |=1$ with $\phi (x_1)\neq \phi (x_2)$ there exists a set $L\subset K$ and an arc $J\subset \mathbb{T}$, $\phi (x_1),\phi (x_2)\in J$ such that $J\subset \phi (L)$. Observe that the constant function $h(x)=1$, $x\in K$, is recurrent for $M_\phi$ since $\Rec(M_\phi )=C(K)$. Thus, there exists a sequence of positive integers $(k_n)_{n\in\mathbb N}$ such that $\sup_{x\in K}|{\phi (x)}^{k_n}-1|\to 0$. It follows that $\sup_{z\in J}|z^{k_n}-1|\to 0$, which is clearly a contradiction. Hence, there exists $a\in \mathbb T$ such that $\phi (x)=a$ for every $x\in K$. 
\end{proof}

\subsection{Multiplication operators on Hilbert spaces of analytic functions on domains of \texorpdfstring{${\mathbb C}^n $}{Cn}.}

We fix a non-empty open connected set $\Omega$ of $\mathbb{C}^n$, $n\in \mathbb{N}$, and $H$ a Hilbert space of holomorphic functions on $\Omega $ such that: 
\begin{itemize}
	\item[-] $H\neq \{ 0\}$, and 
	\item [-]for every $z\in \Omega$, the point evaluation functionals $f\to f(z)$, $f\in H$, are bounded. 
\end{itemize}
Every complex valued function $\phi:\Omega \to \mathbb{C}$ such that the pointwise product $\phi f$ belongs to $H$ for every $f\in H$ is called \emph{a multiplier of $H$}. In particular $\phi$ defines the multiplication operator $M_{\phi}:H\to H$ in terms of the formula 
\begin{align*}
	M_{\phi }(f)=\phi f,\quad f\in H. 
\end{align*}
By the boundedness of point evaluations along with the closed graph theorem it follows that $M_{\phi}$ is a bounded linear operator on $H$. It turns out that under our assumptions on $H$, every multiplier $\phi $ is a bounded holomorphic function, that is $\| \phi \|_{\infty}:=\sup_{z\in \Omega }|\phi (z)|<+\infty $. In particular we have that $\| \phi \|_{\infty }\leq \| M_{\phi}\| $; see \cite{GoSh}.

The recurrent properties of multiplication operators on Hilbert spaces of analytic functions as above are contained in the following theorem, proved in \cite{CP}: 

\begin{theorem}\label{t.adjoints}
	Suppose that every non-constant bounded holomorphic function $\phi $ on $\Omega $ is a multiplier of $H$ such that $\| M_{\phi}\| =\| \phi \|_{\infty }$. Then for each such $\phi$ the following hold. 
	\begin{itemize}
		\item [(i)] The multiplication operator $M_\phi$ is not recurrent. 
		\item [(ii)] The adjoint $M_\phi ^*$ is recurrent if and only if it is hypercyclic if and only if $\phi(\Omega)\cap \mathbb T \neq \emptyset$. 
	\end{itemize}
\end{theorem}

\begin{remark} It is not hard to see that under the hypothesis of the previous theorem $M_\phi ^*$ is never rigid. Indeed, every value $\phi(z)$, $z\in \Omega$, is an eigenvalue of $M_\phi ^*$; see \cite{GoSh}. By Proposition \ref{p.rigspec} we must have $|\phi(z)|\leq 1$. On the other hand if $M_\phi ^*$ is rigid then by the previous proposition $M_\phi ^*$ is hypercyclic and then we necessarily have that $\|M_\phi ^*\|=\|M_\phi\|=\|\phi\|_\infty>1$. Since these two conditions are mutually exclusive we see that $M_\phi ^*$ is not rigid.

On the other hand, again under the assumptions of Theorem~\ref{t.adjoints}, the multiplication operator $M_{\psi }$ is recurrent if and only if $M_\psi ^*$ is recurrent if and only if $\psi$ is equal to a unimodular constant, everywhere on $\Omega$. In the latter case $M_{\psi}, M_\psi ^*$ are, in fact, uniformly rigid. 
\end{remark}

\subsection{Multiplication operators on Banach spaces of holomorphic functions in the unit disk.}

Let $X$ be a a non-trivial Banach space of functions holomorphic in the unit disk $\mathbb D$. A function $\phi\in H(\mathbb D)$ is said to be a (pointwise) multiplier of $X$ into $X$ if $\phi f\in X$ for every $f\in X$. Then the multiplication operator $M_{\phi }:X\to X$ is defined by $M_{\phi }f=\phi f$ for every $f\in X$. Assuming in addition that each point-evaluation functional is bounded on $X$, it follows that: (i) $M_{\phi }$ is a bounded operator, (ii) $\phi $ is bounded on $\mathbb D$, i.e. $\phi \in H^{\infty}(\mathbb D)$ and (iii) $\| \phi \|_{\infty}\leq \| M_{\phi } \| $, see for instance \cite{ADMV}.

\subsection{Multiplication operators on Hardy and Bergman spaces.}

For $1\leq p<\infty $ the Hardy space $H^p$ consists of all holomorphic functions $f$ in the unit disk $\mathbb D$ such that 
\begin{align*}
	\| f\|_{H^p(\mathbb D)}\coloneqq \sup_{0\leq r<1}{\left( \frac{1}{2\pi }\int_0^{2\pi }|f(re^{i\theta })|^pd\theta \right) }^{1/p}<+\infty . 
\end{align*}
Equipped with this norm $H^p(\mathbb D)$ becomes a Banach space and for every $f\in H^p$ we have that 
\begin{align}
	\label{e.growthHp} |f(z)|\leq \frac{\| f\|_{H^p(\mathbb D)}}{(1-|z|^2)^{1/p}}, \quad z\in \mathbb D;
\end{align}
see \cite{Duren}. The last estimate implies that all point-evaluation functionals are bounded.

For $1\leq p<\infty $ the Bergman space $A^p(\mathbb D)$ consists of all holomorphic functions $f$ in the unit disk $\mathbb D$ such that 
\begin{align*}
	\| f\|_{A^p(\mathbb D)}\coloneqq \left( \int_{\mathbb D}|f(z)|^pdA(z) \right)^{1/p} <\infty , 
\end{align*}
where $dA(z)=\frac{1}{\pi}dxdy$ is the normalized area measure in $\mathbb D$ with $A(\mathbb D)=1$. Then $A^p$ becomes a Banach space and for every $f\in A^p(\mathbb D)$ we have that 
\begin{align}\label{e.growthAp}
	|f(z)|\leq \frac{\| f\|_{A^p (\mathbb D)}}{(1-|z|^2)^{2/p}}, \quad z\in \mathbb D , 
\end{align}
which in turn implies that all point-evaluation functionals of $A^p(\mathbb D)$ are bounded. For the growth estimate above see for example \cite{ADMV}.

We will also consider the Dirichlet space on the unit disc, denoted by $\mathcal D$, and consisting of all the holomorphic functions $f$ on the unit disc such that 
\begin{align*}
	\|f\|_{\mathcal D} ^2 \coloneqq |f(0)|^2 +\int_{\mathbb D} |f'(z)|^2 dA(z)<+\infty. 
\end{align*}

Finally, we consider the Bloch space on the unit disc, denoted by $\mathcal B$, and consisting of all functions $f$, holomorphic on $\mathcal D$, such that 
\begin{align*}
	\|f\|_{\mathcal B}\coloneqq |f(0)|+\sup_{z\in\mathbb D} (1-|z|^2)|f'(z)|^2<+\infty. 
\end{align*}
It is well known that functions both in the Dirichlet space $\mathcal D$ as well as in the Bloch space $B$ satisfy growth estimates similar to \eqref{e.growthHp} and thus the corresponding point evaluation functionals are bounded. See for example \cite{ADMV}.

\begin{theorem}
	Let $1\leq p<\infty$. Consider a multiplier $\phi$ of $H^p(\mathbb D)$ into $H^p(\mathbb D)$ and the corresponding multiplication operator $M_{\phi}:H^p(\mathbb D)\to H^p(\mathbb D)$. The following are equivalent. 
	\begin{itemize}
		\item[(i)] $M_\phi $ is recurrent. 
		\item[(ii)] $M_\phi$ is rigid. 
		\item[(iii)] $M_\phi$ is uniformly rigid. 
		\item [(iv)] There exists $a\in \mathbb T$ such that $\phi (x)=a$ for every $x\in \mathbb D$. 
	\end{itemize}
\end{theorem}

\begin{proof}
	We only have to prove that (i) implies (iv) since all the other implications are trivial. Consider $f\in \Rec(M_{\phi })$ so that $f$ is not identically zero. We have 
	\begin{align*}
		|\phi (z)^{n}f(z)-f(z)|\leq \frac{ \| M_{\phi }^nf-f\|_{H^p(\mathbb D)} }{(1-|z|^2)^{1/p}}, \quad z\in \mathbb D . 
	\end{align*}
	There exists a strictly increasing sequence of positive integers $(k_n)_{n\in\mathbb N}$ such that $\|M_\phi ^{k_n}f-f\|_{H^p(\mathbb D)} \to 0$ as $n\to+\infty$. Therefore, $\phi (z)^{k_n}f(z)\to f(z)$ uniformly on compact subsets of $\mathbb D$ and since $f$ is not identically zero we conclude that $\phi (z)^{k_n}\to 1$ uniformly on any open disk $B$ such that $\overline{B} \subset \mathbb D$. It follows that $|\phi (z) |=1$ for every $z\in B$ and by the maximum modulus principle and analytic continuation we conclude that $\phi $ is a unimodular constant on $\mathbb D$. 
\end{proof}

\begin{theorem}
	Let $X$ be either the Bergman space $A^p(\mathbb D)$, $1\leq p<+\infty$, the Dirichlet space $\mathcal D$ or the Bloch space $\mathcal B$. Consider a multiplier $\phi$ of $X$ into $X$ and the corresponding multiplication operator $M_{\phi }:X\to X$. The following are equivalent. 
	\begin{itemize}
		\item[(i)] $M_\phi $ is recurrent. 
		\item[(ii)] $M_\phi$ is rigid. 
		\item[(iii)] $M_\phi$ is uniformly rigid. 
		\item [(iv)] There exists $a\in \mathbb T$ such that $\phi (x)=a$ for every $x\in \mathbb D$. 
		\item [(v)] $\phi$ is an isometric multiplier of $X$. 
	\end{itemize}
\end{theorem}

\begin{proof}
	We give the proof in the case of the Bergman space $A^p(\mathbb D)$. For the other two cases the proof is similar using the corresponding growth estimates instead of \eqref{e.growthAp}. We first prove that (i) implies (iv). Using \eqref{e.growthAp} and arguing as in the proof of the previous theorem the desired implication follows. In \cite{ADMV} it is proved that (iv) is equivalent to (v). Now, all other implications are trivial. This completes the proof of the theorem. 
\end{proof}

\subsection{Multiplication operators on \texorpdfstring{$L^2(X,\mu)$}{L2(X,mu)}.}

Let $(X,\mu)$ be a measure space where $\mu$ is a non-negative finite Borel measure. For $\phi\in L^\infty(X,\mu)$ we define the multiplication operator $M_\phi:L^2(X,\mu)\to L^2(X,\mu)$ as $M_\phi(f)=\phi f$. Obviously $M_\phi$ is a bounded linear operator with $\|M_\phi\|\leq \|\phi\|_{L^\infty(X,\mu)}$. We have the following.

\begin{theorem}
	\label{t.multcond} Let $M_\phi:L^2(X,\mu)\to L^2(X,\mu)$ be the multiplication operator with symbol $\phi$. 
	\begin{itemize}
		\item[(A)] The following are equivalent. 
		\begin{itemize}
			\item [(i)] $M_\phi$ is recurrent. 
			\item [(ii)] $M_\phi$ is rigid. 
			\item [(iii)] There exists a strictly increasing sequence of positive integers $(k_n)_{n\in\mathbb N}$ such that 
			\begin{equation}
				\label{e.cond} \phi(x)^{k_n}\to 1\quad\text{as}\quad n\to +\infty,\quad \mu \text{-almost everywhere}. 
			\end{equation}
		\end{itemize}
		\item[(B)] The following are equivalent. 
		\begin{itemize}
			\item [(i)] $M_\phi$ is uniformly rigid. 
			\item [(ii)] There exists a strictly increasing sequence of positive integers $(k_n)_{n\in\mathbb N}$ such that 
			\begin{equation*}
				\|\phi^{k_n}-1\|_{L^\infty(X)} \to 0\quad\text{as}\quad n\to +\infty. 
			\end{equation*}
		\end{itemize}
	\end{itemize}
	In particular if $M_\phi$ is recurrent then $|\phi|=1$ $\mu$-almost everywhere. 
\end{theorem}

\begin{proof}
	We begin by showing the equivalences in (A). Let $\phi\in L^\infty(X,\mu)$ satisfy \eqref{e.cond}. Applying Lebesgue's dominated convergence theorem we immediately get that $M_\phi$ is rigid.
	
	We now prove that (i) implies (iii) so we assume that $M_\phi$ is recurrent. We first show that $|\phi|=1$ for $\mu$-almost every $x\in X$. For any positive integer $m$ let $E_m\coloneqq\{x\in X:|\phi(x)|>1+\frac{1}{m}\}$ and assume, for the sake of contradiction, that $\mu(E_m)>0$ for some $m$. Since $M_\phi$ is recurrent and $E_m$ has positive measure there exists a recurrent vector $f\in L^2(X,\mu)$ which is not identically zero on $E_m$. Then for any strictly increasing sequence of positive integers $(l_n)_{n\in\mathbb N}$ we have 
	\begin{align*}
		\|M_\phi ^{l_n} f-f\|_{L^2(X,\mu)}\geq\bigg( \int_{E_m}|\phi(x)^{l_n}-1|^2|f(x)|^2d\mu(x)\bigg)^\frac{1}{2}\geq \Abs{\big(1+{1}/{m}\big)^{l_n}-1} \int_{E_m} |f(x)|^2\to +\infty 
	\end{align*}
	as $n\to +\infty$ which is a contradiction since $f$ was a recurrent vector. Thus $\mu(E_m)=0$ for all positive integers $m$ which shows that $\mu(\{x\in X:|\phi(x)|>1\})=0.$ A similar argument shows that $\mu(\{x\in X:|\phi(x)|<1\})=0$ so that $|\phi|=1$, $\mu$-almost everywhere.
	
	Observe now that $M_\phi$ is unitary since $|\phi|=1$ $\mu$-almost everywhere and this implies that every $f\in L^2(X,\mu)$ is a recurrent vector for $M_\phi$. In particular the constant function $1\in L^2(X,\mu)$ is a recurrent vector thus there exists a strictly increasing sequence of positive integers $(k_n)_{n\in\mathbb N}$ such that 
	\begin{align*}
		\|\phi^{k_n}-1\|_{L^2(X,\mu)}\to 0\quad\text{as}\quad n\to +\infty. 
	\end{align*}
	By standard arguments we get the existence of a subsequence which we also call $(k_n)_{n\in\mathbb N}$ such that $\phi^{k_n}(x)\to 1$, for $\mu$-almost every $x\in X$. For (B) it is enough to note that for every $n\in\mathbb N$ we have $\|M_\phi ^n-I\|=\|\phi^n-1\|_{L^\infty(X)}$. 
\end{proof}

We would like to have a more hands-on characterization of the functions $\phi$ that give rise to recurrent operator $M_\phi$. For example, it is straightforward to see that for every constant function $\phi$ such that $\phi(x)=a$ for some $a\in \mathbb T$, $\mu$-almost everywhere, the operator $M_\phi$ is recurrent; but are these the only ones? It turns out that the answer is \emph{no} in general, and that one cannot expect a characterization of the symbols $\phi$ that give recurrent multiplication operators on $L^2(X,\mu)$, in the case of a general measure space $(X,\mu)$. We illustrate this by two examples. 

\begin{example}
	Let $X=\mathbb T$ and $d\mu=d\theta$ be the Lebesgue measure on the circle. By the previous analysis we immediately restrict our attention to measurable functions $\phi:\mathbb T\to \mathbb T$. It is clear that constant unimodular functions as well as functions that take a finite number of values on $\mathbb T$ give rise to recurrent operators. Instead of showing this, which is an easy exercise, we present a slightly more general example below. 
\end{example}

\begin{example}
	\label{e.countable} Let $\phi:\mathbb T\to \mathbb T$ be a measurable function such that the set $\phi(\mathbb T)=\{\phi(t):\ t\in\mathbb T\}$ is countable. Then the multiplication operator $M_\phi:L^2(\mathbb T,d\theta)\to L^2(\mathbb T,d\theta)$ is recurrent. Indeed, by Theorem~\ref{t.multcond} it suffices to find a strictly increasing sequence of positive integers $(k_n)_{n\in \mathbb N}$ such that $\phi(x)^{k_n}\to 1$, $\mu$-almost everywhere. Let us write $\phi(\mathbb T)=\{e^{2\pi i \theta_k}: k=1,2,\ldots \}$ with $(\theta_k)_{k\in\mathbb N}\subset \R$. We inductively construct a strictly increasing sequence $(k_n)_{n\in\mathbb N}$ such that 
	\begin{align*}
		|e^{2\pi i k_n\theta_j}-1|<1/n\quad\text{for all}\quad j=1,2\ldots,n. 
	\end{align*}
	Therefore $\phi(x)^{k_n}\to 1$ as $n\to +\infty$ for all $x\in\mathbb T$, so we are done. 
\end{example}

\begin{example}
	There exists a non-constant continuous function $\phi:[0,1] \to \mathbb T$ such that $M_\phi$ is recurrent on $L^2([0,1],d\theta)$. To see this, let $C$ denote the triadic Cantor set on $[0,1]$ and $f:[0,1]\to [0,1]$ be the \emph{Cantor-Lebesgue} function as constructed for example in \cite{WZ}*{page 35}. Since $f$ is continuous the function $\phi(t)\coloneqq e^{2\pi if(t)}$ is a well defined, continuous function. Furthermore the set $\{\phi(t):t\in\mathbb [0,1]\setminus C\}$ is countable. As in Example \ref{e.countable} we can find a strictly increasing sequence of positive integers $(k_n)_{n\in\mathbb N}$ such that $\phi(t)^{k_n}\to 1$ for every $t\in [0,1]\setminus C$, that is, almost everywhere since $|C|=0$. By Theorem~\ref{t.multcond} we get that $M_\phi$ is recurrent. 
\end{example}

\begin{example}
	Let us denote by $C(\mathbb T)$ the set of continuous functions $f:\mathbb T\to \mathbb C$ and set $S\coloneqq\{f\in\mathbb C(\mathbb T):|f(z)|=1\text{ for all }z\in\mathbb T \}$. Recall that a closed set $K\subset \mathbb T$ is called a \emph{Kronecker set} if for every $f\in S$ and every $\epsilon>0$ there exists a positive integer $n$ such that 
	\begin{align*}
		\sup_{z\in K}|f(z)- z^n|<\epsilon. 
	\end{align*}
	That is, $K$ is Kronecker if we can approximate every continuous unimodular function on $\mathbb T$ by characters, uniformly on $K$. For a detailed discussion of Kronecker sets see for example \cite{Rudin}. By Theorem~\ref{t.multcond} we get that the function $\chi_1(z)\coloneqq z$, $z\in\mathbb T$, gives rise to a uniformly rigid operator $M_{\chi_1}:L^2(K,d\theta)\to L^2(K,d\theta)$, whenever $K$ is a Kronecker set. Automatically we get that all the characters $\chi_m(z)\coloneqq z^m$ give uniformly rigid operators on $L^2(K,d\theta)$ by using the simple estimate $|(z^m)^n-1|\leq m|z^n-1|,$ whenever $z\in\mathbb C$ with $|z|=1$.
\end{example}

\section{Unitary, Normal, Hyponormal and \texorpdfstring{$m$}{m}-isometric operators} \label{s.unitary} In this section we fix $H$ to be a \emph{separable} Hilbert space over $\C$. The main theme of this paragraph is that if a recurrent operator on a Hilbert space has ``sufficient structure'' then it reduces to a unitary operator. A first instance of this heuristic is contained in the following proposition.

Recall that an operator $T:H\to H$ is \emph{normal} if $TT^*=T^*T$, where $T^*$ is the Hilbert space adjoint of $T$. 

\begin{proposition}
	\label{p.normal} If a normal operator $T:H\to H$ is recurrent then $T$ is unitary. 
\end{proposition}

\begin{proof}
	Since $T$ is normal there exist a finite positive Borel measure $\mu $ on the spectrum $\sigma (T)$ and a function $\phi \in L^{\infty }(\mu)$ such that $T$ is unitarily equivalent to $M_{\phi}:L^2(\mu )\to L^2(\mu )$, where $M_{\phi }f=\phi f$ for $f\in L^2(\mu)$. By Theorem~\ref{t.multcond} it follows that $|\phi |=1$ $\mu$-almost everywhere on $\sigma (T)$, therefore $\| M_{\phi }\|={\| \phi \| }_{\infty }=1$. Since $T$ is unitarily equivalent to $M_{\phi }$ we get that $\| T\| \leq 1$ and by Proposition \ref{p.bimpliesunit}, (i), we conclude that $T$ is unitary. 
\end{proof}

It is classical that a normal operator $T:H\to H$ is cyclic if and only if there exists a finite positive Borel measure $\mu $ on the spectrum $\sigma (T)$ so that $T$ is unitarily equivalent to $M_z:L^2(\mu)\to L^2(\mu)$, where $M_zf=zf$ for $f\in L^2(\mu)$. See for example \cite{Con}. Using this, Theorem~\ref{t.multcond} and the fact that recurrence is preserved by unitary equivalence we immediately get the following corollary. 

\begin{corollary}
	Let $T:H\to H$ be a normal operator. Then $T$ is recurrent and cyclic if and only if there exist a finite positive Borel measure $\mu $ on the spectrum $\sigma (T)$ and a strictly increasing sequence of positive integers $(k_n)_{n\in\mathbb N}$ such that $T$ is unitarily equivalent to $M_z:L^2(\mu)\to L^2(\mu)$, where $M_zf=zf$ for $f\in L^2(\mu)$ and 
	\begin{align*}
		\int_{\sigma (T)}|z^{k_n}-1|^2d\mu (z) \to 0. 
	\end{align*}
\end{corollary}

We now turn our attention to \emph{hyponormal} operators, namely, operators $T:H\to H$ having the property $\| T^*h\| \leq \| Th\|$ for every $h\in H$. The next proposition extends Proposition \ref{p.normal}. Its proof necessarily avoids the spectral theorem for normal operators. Instead, it relies on certain inequalities for orbits of hyponormal operators established by Bourdon in \cite{Bourdon}. 

\begin{proposition}
	\label{p.normalimpliesunit} If a hyponormal operator $T:H\to H$ is recurrent then $T$ is unitary. 
\end{proposition}

\begin{proof}
	Take a non-zero vector $h \in \Rec(T)$. Then there exists a strictly increasing sequence of positive integers $(k_n)_{n\in\mathbb N}$ such that $T^{k_n}h\to h$. Observe that $T^nh\neq 0$ for every positive integer $n$. It follows that 
	\begin{align*}
		\frac{\| T^{k_n+1}h\| }{\| T^{k_n}h\| }\to \frac{ \| Th\| }{\| h\| }. 
	\end{align*}
	From the last and the continuity of $T$ we get that
	\begin{align*}
		\frac{\| T^{k_n+m}h\| }{\| T^{k_n+m-1}h\| }\to \frac{ \| T^mh\| }{\| T^{m-1}h\| } 
	\end{align*}
	for every $m=1,2,\ldots$. Since $T$ is hyponormal and $h\notin \Ker( T) $, by \cite{Bourdon}*{Theorem 4.1} and the discussion in \cite{Bourdon}*{p. 350-351}, we have that the limit 
	\begin{align*}
		\lim_{n\to +\infty}\frac{\| T^{n+1}h\| }{\| T^nh\| } 
	\end{align*}
	exists. It readily follows that 
	\begin{align*}
		\frac{ \| Th\| }{\| h\| }=\frac{ \| T^2h\| }{\| Th\| }=\cdots =\frac{ \| T^mh\| }{\| T^{m-1}h\| }=\cdots . 
	\end{align*}
	The above equalities and an easy induction argument imply that
	\begin{align*}
		\| T^nh\| =\frac{{\| Th\| }^n}{\| h\|^{n-1} } \quad \text{for every} \quad n=1,2,\ldots . 
	\end{align*}
	The fact that $\| T^{k_n}h\| \to \| h\| \neq 0$ combined with the previous equality implies that $\| Th \| =\| h\|$. Hence $\| Th\|= \| h\| $ for every $h \in \Rec(T)$ and since $\Rec(T)$ is dense we conclude that $\| Th\| =\| h\| $ for every $h\in H$. Therefore $\| T\| =1$ and by Proposition \ref{p.bimpliesunit}, (i), we conclude that $T$ is unitary. 
\end{proof}

An amusing application of the notions in this paragraph is the observation that, if $T:H\to H$ is hyponormal with $\| T\| >1$ and $h\in H$ is a non-zero recurrent vector for $T$ then $T$ has a non-trivial invariant subspace. Indeed, by the proof of \cite{Bourdon}*{Theorem 4.1} and the proof of the previous proposition it follows that 
\begin{align*}
	\lim_{n\to +\infty}\| T^nh\|^{1/n}=\lim_{n\to +\infty}\frac{\| T^{n+1}h\| }{\| T^nh\| } =\frac{ \| Th\| }{\| h\| }=1< \| T\| . 
\end{align*}
Now \cite{Bourdon}*{Proposition 4.6} implies that $T$ has a non-trivial invariant subspace. The existence of non-trivial invariant subspaces for hyponormal operators is, in general, an open problem. For partial results see \cite{ScBr}.

We recall the notion of an $(m,p)$-isometry in general Banach spaces, introduced by Bayart in \cite{B1}. 

\begin{definition}
	Let $T:X\to X$ be an operator and let $m\in \mathbb{N}$, $p\in [1,+\infty)$. $T$ is called an $(m,p)$-isometry if 
	\begin{align*}
		\sum_{k=0}^{m}(-1)^{m-k}\binom{m}{k}\| T^kx\|^p=0 , 
	\end{align*}
	for every $x\in X$. $T$ is called an $m$-isometry if it is an $(m,p)$-isometry for some $p\in [1,+\infty )$. 
\end{definition}

\begin{proposition}
	If the operator $T:X\to X$ is an $m$-isometry and recurrent then $T$ is a surjective isometry. 
\end{proposition}

\begin{proof}
	Let $y\in \Rec(T)$. Then there exists a strictly increasing sequence of positive integers $(k_n)_{n\in\mathbb N}$ such that $T^{k_n}y\to y$ and hence $T^{k_n+1}y\to Ty$. By \cite{B1}*{Proposition 3.1} the sequence $(\| T^nx\|)_{n\in\mathbb N}$ is eventually increasing for every $x\in X$. Therefore we get $\| Ty\| =\| y\|$. Since $\Rec(T)$ is dense in $X$ we readily see that $\| Tx\| =\| x\|$ for every $x\in X$. So $\| T\| =1$ and by Proposition \ref{p.bimpliesunit}, (i), we conclude that $T$ is a surjective isometry. 
\end{proof}

\section{The \texorpdfstring{$T\oplus T$}{T+T} problem: Product recurrence} \label{s.toplust} The problem which concerns us here is whether $T\oplus T$ is recurrent whenever $T$ is recurrent. The corresponding problem for hypercyclic operators was a long standing question that was only recently settled, in the negative, by de la Rosa and Read \cite{DR}, and also by Bayart and Matheron, \cite{BM2}, in classical Banach spaces. Our first result in this direction is the following. 

\begin{theorem}
	\label{t.prodrec} Let $T:X\to X$ be a recurrent operator and consider the commutant $\{ T \}'$ of $T$, i.e. $\{ T\}'=\{ A: AT=TA \}$. Suppose there exists a subset $\mathcal M$ of $\{ T \}'$ such that the set 
	\begin{align*}
		\{ x\in X: \overline{ \{Ax: A\in \mathcal M \} }=X \} 
	\end{align*}
	is residual in $X$. Then $T\oplus T$ is recurrent. 
\end{theorem}

\begin{proof}
	Since $\Rec(T)$ is $G_{\delta }$ and dense, take $x\in \Rec (T)\cap \{ x\in X: \overline{ \{Ax: A\in \mathcal M \} }=X \}$ and consider $A\in \mathcal M$. There exists a strictly increasing sequence of positive integers $(k_n)_{n\in\mathbb N}$ such that $T^{k_n}x\to x$ as $n\to+\infty$. Since $A$ and $T$ commute we get that $T^{k_n}Ax\to Ax$. It follows that 
	\begin{align*}
		\{ x\oplus Ax: x\in \Rec(T)\cap \{ x\in X: \overline{ \{Ax: A\in \mathcal M \} }=X \},\ A\in \mathcal M \} \subset \Rec(T\oplus T). 
	\end{align*}
	Observe now that the left hand side set in the previous inclusion is dense in $X\oplus X$ which in turn implies that $T\oplus T$ is recurrent. 
\end{proof}

\begin{corollary}
	Let $T:X\to X$ be an operator. If $T$ is cyclic and recurrent then $T\oplus T$ is recurrent. 
\end{corollary}

\begin{proof}
	Proposition \ref{p.pspectrum} implies that $ \sigma_p(T^*) ^{\circ}=\emptyset$ and in view of \cite{H}*{Theorem 1} we conclude that the set of cyclic vectors for $T$ is $G_{\delta }$-dense in $X$. Now, the conclusion follows by applying Theorem \ref{t.prodrec} for $\mathcal M :=\{ p(T):  p \  \text{polynomial} \}$ 
\end{proof}

Trivially if $T$ is rigid then $T\oplus T$ is recurrent. Thus, according to the following proposition, $T\oplus T$ is recurrent whenever $T$ is a recurrent unitary operator on a complex Hilbert space. 

\begin{proposition}
	\label{p.urecequalsrig} Let $H$ be a separable Hilbert space and $U:H\to H$ be a unitary operator. The following are equivalent. 
	\begin{itemize}
		\item[(i)] $U$ is recurrent. 
		\item[(ii)] $U$ is rigid. 
	\end{itemize}
\end{proposition}

\begin{proof}
	We only have to prove that (i) implies (ii), so suppose that $U$ is recurrent. By the spectral theorem for unitary operators on separable Hilbert spaces, see \cite{RaRo}*{Theorem 1.6}, there exists a measure space $(X,\mathcal X,\mu)$, where $\mu$ is a non-negative finite Borel measure, a unitary map $\Phi:H\to L^2(X,\mu)$, and a function $u\in L^\infty(X,\mu)$ with $|u|=1$ such that the operator $\Phi\circ U\circ \Phi^{-1}:L^2(X,\mu)\to L^2(X,\mu)$ is the multiplication operator
\begin{align*}
 M_u f \coloneqq u f, \quad \forall f\in L^2(X,\mu).
\end{align*}
Since $U$ is recurrent we immediately get that $M_u$ is recurrent. By Theorem~\ref{t.multcond} we get that $M_u$ is rigid and thus $U:H\to H$ is rigid on $H$. \end{proof}

\begin{remark}
	If $T:H \to H$ is a power bounded operator acting on a separable Hilbert space then, by Proposition \ref{p.bimpliesunit}, we get that $T$ is similar to a unitary operator. Thus Proposition \ref{p.urecequalsrig} remains valid if the hypothesis that $T$ is unitary is replaced by the hypothesis that $T$ is power bounded.
\end{remark}

The analogue of the previous proposition for a general Banach space $X$ is, to the best of our knowledge, an open problem:

\begin{question}
	\label{q.suriso} Let $T:X\to X$ be a surjective isometry. Is it true that $T$ is recurrent if and only if $T$ is rigid? 
\end{question}

Finally, we indicate a connection of the last question with another open problem in Operator Theory. An operator $T:X\to X$, acting on a separable Banach space $X$, is called \emph{orbit reflexive} if the only operators $S:X\to X$ such that $Sx\in \overline{\Orb(x,T)}$ for every $x\in X$ are those in $\overline{\{I,T,T^2,\ldots\}} ^{\operatorname{SOT}}$. In \cite{HNRR} it was shown that ``many'' Hilbert-space operators are orbit reflexive, for instance, normal, compact, algebraic operators and contractions. Examples of Hilbert-space operators that are \emph{not} orbit-reflexive only recently appeared in \cites{GrivRog,MV}. It is an open question whether every power bounded operator $T:X\to X$ is orbit reflexive. See for example \cite{HIH}. Observe that a positive answer to this question would imply an affirmative answer to Question \ref{q.suriso} above.

For every recurrent operator $T$ appearing in the present paper, $T\oplus T$ is also recurrent. However, the following questions seems to be open. 

\begin{question}
	Let $T:X\to X$ be a recurrent operator. Is it true that the operator $T\oplus T$ is recurrent on $X\oplus X$? 
\end{question}


\end{document}